\documentclass[pdflatex,sn-mathphys-num]{sn-jnl}

\usepackage{graphicx}%
\usepackage{multirow}%
\usepackage{amsmath,amssymb,amsfonts}%
\usepackage{amsthm}%
\usepackage{mathrsfs}%
\usepackage[title]{appendix}%
\usepackage{xcolor}%
\usepackage{textcomp}%
\usepackage{manyfoot}%
\usepackage{booktabs}%
\usepackage{algorithm}%
\usepackage{algorithmicx}%
\usepackage{algpseudocode}%
\usepackage{listings}%
\usepackage{setspace}

\theoremstyle{thmstyleone}%
\newtheorem{theorem}{Theorem}
\newtheorem{proposition}{Proposition}

\newtheorem{remark}{Remark}%
\newtheorem{lemma}{Lemma}
\newtheorem{corollary}{Corollary}
\newtheorem{assumption}{Assumption}%

\usepackage{comment}

\def\D{\displaystyle}				
\def\ol{\overline}
\def\ul{\underline}

\def\wt{\tilde}

\def\Lc{{\cal L}}
\def\Oc{{\cal O}}

\def\fct#1{\mathop{\rm #1}}	                

\def\angle{\fct{a}}

\def\trec{\fct{trec}}

\def\AS{\mathbf{(A_{S})}}
\def\AL{\mathbf{(A_{L})}}
\def\AN{\mathbf{(A_{N})}}

\def\ASn{\mathbf{A_{S}}}
\def\ALn{\mathbf{A_{L}}}
\def\ANn{\mathbf{A_{N}}}

\def\argmaxx{\fct{argmax}}
\def\argminx{\fct{argmin}}
\def\trec{\fct{trec}}
\def\rec{\fct{trial}}

\def\tar{\fct{tar}}

\def\rec{\fct{rec}}

\def\btr{{\color{blue}{$\blacktriangleright$}} }

\usepackage[normalem]{ulem}

\usepackage{enumitem}
\usepackage{hyperref}
\newcommand{\orcidicon}[1]{%
	\href{https://orcid.org/#1}{\includegraphics[height=1.5ex]{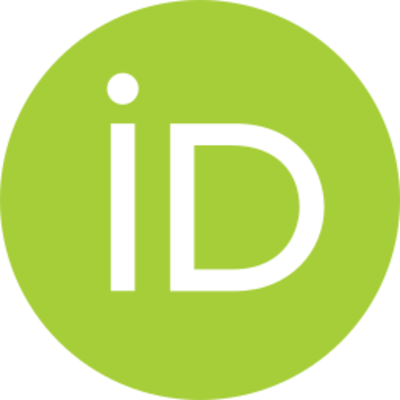}}%
}

\usepackage{tabularx}

\usepackage{tikz}
\usepackage{pgfplots}
\usepgfplotslibrary{groupplots}
\pgfplotsset{compat=1.18}

\definecolor{EPBlue}{RGB}{31,119,180}
\definecolor{EPOrange}{RGB}{255,127,14}
\definecolor{EPGreen}{RGB}{44,160,44}
\definecolor{EPRed}{RGB}{214,39,40}
\definecolor{EPPurple}{RGB}{148,103,189}
\definecolor{EPBrown}{RGB}{140,86,75}

\pgfplotsset{
  daesnoiseaxis/.style={
    width=6.9cm,height=5.0cm,scale only axis,
    xmode=log,log basis x=10,xmin=1e-5,xmax=1e2,
    enlarge x limits=false,
    xtick={1e-5,1e-4,1e-3,1e-2,1e-1,1,10,100},
    xticklabels={$10^{-5}$,$10^{-4}$,$10^{-3}$,$10^{-2}$,$10^{-1}$,$1$,$10$,$100$},
    title style={font=\small,at={(0,1.03)},anchor=south west},
    ymin=0,enlarge y limits=0.08
  },
  profileaxis/.style={
    title style={font=\small,at={(0,1.03)},anchor=south west}
  },
  curveA/.style={draw=EPBlue,solid,mark=o,line width=1.2pt,mark size=1.7pt,mark options={solid,fill=white,line width=.55pt}},
  curveB/.style={draw=EPOrange,dashed,mark=square,line width=1.2pt,mark size=1.7pt,mark options={solid,fill=white,line width=.55pt}},
  curveC/.style={draw=EPGreen,densely dotted,mark=diamond,line width=1.2pt,mark size=1.8pt,mark options={solid,fill=white,line width=.55pt}},
  curveD/.style={draw=EPRed,dashdotted,mark=triangle,line width=1.2pt,mark size=1.9pt,mark options={solid,fill=white,line width=.55pt}},
  curveE/.style={draw=EPPurple,dotted,mark=asterisk,line width=1.2pt,mark size=1.9pt,mark options={solid,fill=white,line width=.55pt}},
  curveF/.style={draw=EPBrown,dashed,mark=pentagon,line width=1.2pt,mark size=1.8pt,mark options={solid,fill=white,line width=.55pt}}
}

\begin{document}

\newpage
\pagenumbering{arabic}
\setcounter{page}{1}
\title[A Direction Adaptation Evolution Strategy for Noisy Derivative-Free Optimization]{A Direction Adaptation Evolution Strategy for Noisy Derivative-Free Optimization}
\author[1]{\fnm{Morteza} \sur{Kimiaei}\orcidicon{0000-0002-7973-3770}}	
\author[2]{\fnm{Mahsa} \sur{Yousefi}\orcidicon{0000-0002-2937-9654}}		
\affil[1]{\orgdiv{Faculty of Mathematics}, \orgname{University of Vienna}, \orgaddress{\street{Oskar-Morgenstern-Platz 1}, \city{Vienna}, \postcode{A-1090}, \country{Austria}}}
\email{\small morteza.kimiaei@univie.ac.at}

\affil[2]{\orgdiv{Department of Industrial Engineering}, \orgname{University of Florence}, \orgaddress{\street{Viale G.B. Morgagni 40}, \city{Florence}, \postcode{50134}, \country{Italy}}}
\email{\small mahsa.yousefi@unifi.it}

    \abstract{In noisy derivative-free optimization (DFO), methods based on the matrix adaptation evolution strategy (\texttt{MAES}) often mistakenly identify points with spuriously good function values, undermining selection and leading to unreliable search directions. To address this, we develop a novel direction adaptation evolution strategy (\texttt{DAES}) that targets the algorithmic phases prone to noise propagation, combining selected-population search with complexity analysis. Fixing the transformation matrix to the identity, \texttt{DAES} employs a structured direction-generation mechanism consisting of symmetric sampling, joint sorting--selection of noisy values, partitioning of all sampled directions into three ranked groups, and construction of a heuristic search direction. Along this direction, a noisy derivative-free sufficient descent condition with extrapolation is applied. Using a probabilistic angle condition preserved through our proposed framework, we obtain, to the best of our knowledge, the first high-probability evaluation-complexity bounds for a noisy \texttt{MAES}-type method across nonconvex, convex, and strongly convex objectives, under bounded noise. Finally, experiments on 230 \texttt{BARON} \texttt{BCP} problems across eight noise levels compare \texttt{DAES} with five DFO solvers, notably \texttt{MADFO}, showing greater robustness at intermediate-to-small budgets.}


\maketitle

\section{Introduction}\label{IntroSec}
We consider the unconstrained derivative-free optimization (DFO) problem 
	\begin{equation}\label{e.prob}
		\min_{x \in \mathbb{R}^n} f(x),
	\end{equation}
where the objective function $f : \mathbb{R}^n \to \mathbb{R}$ is smooth but not directly accessible. Instead, we assume access to a \emph{noisy oracle} that, for a given point $x \in \mathbb{R}^n$, returns a noisy function value $\tilde{f}(x)$, contaminated by an unknown noise term $\tilde{f}(x) -f(x)$. We make no assumptions regarding the explicit form of $f(x)$, the availability of its gradient $g(x) := \nabla f(x)$, Lipschitz continuity, or the statistical properties of the noise. This setting is commonly referred to as the \emph{noisy DFO problem}; see, e.g., \cite{audet2017derivative,Conn2009IntroductionTD}.
	
To solve the noisy DFO problem \eqref{e.prob}, numerous deterministic and randomized methods have been proposed, including line search (\texttt{LS}) solvers~\cite{larson2019derivative,Berahas2019,Brilli2024,Giovannelli2022,VRDFON,VRBBO,SSDFO,SDBOX,de2025linesearch}, trust-region (\texttt{TR}) methods~\cite{larson2019derivative,Bandeira2014,BCDFO,powell2002uobyqa,Liuzzi2019}, direct search (\texttt{DS}) methods~\cite{larson2019derivative,Dzahini2025,Gratton2015,NMSMAX,BFO,MCS,Custdio2014,Custdio2018}, and matrix adaptation evolution strategies (\texttt{MAES})~\cite{auger2005restart,beyer2017simplify,MATRS,MADFO}. For advanced DFO methods that combine \texttt{DS} or \texttt{LS} with model-based techniques or with \texttt{MAES}, we refer the reader to~\cite{Audet2022,MATRS,MADFO}.

{\tt DS} methods evaluate the objective function along a finite set of polling directions, accepting trial points that satisfy a forcing-function-based decrease condition (typically proportional to the squared step size). Requiring no derivative information, they adjust step sizes using simple increase--decrease rules. They are robust and effective for low-dimensional problems.

Conversely, several {\tt LS} methods explore descent along prescribed directions, differing mainly in how they enforce sufficient decrease. While some variants approximate directional derivatives using Armijo/Wolfe conditions, others rely on forcing functions akin to {\tt DS} methods—highlighting the close link between derivative-free {\tt LS} and {\tt DS} schemes. Additionally, these variants employ extrapolation steps to bypass saddle points or local maximizers and rapidly converge toward an approximate minimizer.

{\tt MAES} consists of three phases: mutation, selection, and recombination. In the first phase, candidate solutions are generated by sampling random directions and applying a transformation matrix, which requires a specific construction and matrix--vector multiplications. In the second phase, mutation points (candidates) are ranked according to their noisy objective function values, and the corresponding mutation directions associated with the lowest values are retained. The last phase computes a weighted average of the selected directions to form a direction by which a new recombination point is generated. Finally, the transformation matrix and step-sizes are updated. In noisy regimes, an \texttt{MAES}-based algorithm may mistakenly identify points with spuriously good function values, undermining the reliability of the selection process and leading to misleading search directions. Failure in this phase can adversely affect the recombination phase. Furthermore, global and linear convergence are known for certain variants~\cite{Diouane2015ES,Akimoto2022ES,Toure2023ES}, but these results do not establish evaluation-complexity bounds for noisy \texttt{MAES}-type methods. We address these challenges in this paper.

\subsection{Our Contribution}

To address the limitation in \texttt{MAES}-type methods, the central objective of the present paper is twofold: first, to mitigate noise sensitivity through targeted algorithmic modifications across phases with the potential to control noise; and second, to establish complexity guarantees for the resulting method.

For noisy DFO, we propose the direction adaptation evolution strategy (\texttt{DAES}), a new \texttt{MAES}-type method. \texttt{DAES} promotes balanced exploration by sampling symmetric pairs of random directions. It selects a specialized population from multiple generated populations and applies structured within- and cross-group weighting schemes to jointly target instabilities arising from evaluation noise and directional cancellation. Specifically, all mutation directions of the selected population are partitioned into three ranked groups, yielding three directions whose geometric arrangement forms a novel search direction toward promising regions. This redesign replaces covariance matrix adaptation with a structured population-based direction mechanism; indeed, unlike standard \texttt{MAES}, \texttt{DAES} fixes the transformation matrix to the identity, simplifying its theoretical analysis while preserving efficiency. To guarantee sufficient decrease, \texttt{DAES} adopts a forcing function-based condition with an extrapolation strategy.

Under standard assumptions, we establish a probabilistic angle condition for the novel search direction and control per-iteration failure probabilities to guarantee uniform high-probability directional and descent properties. Across nonconvex, convex, and strongly convex settings, we derive complexity bounds matching deterministic rates at the noise-limited accuracy level—to our knowledge, the first for a noisy \texttt{MAES}-type method. Numerically, we evaluate the structurally simple \texttt{DAES} via ablation studies—which favor single-population and extrapolation settings for efficiency—and comparisons against five solvers across 230 \texttt{BCP} problems at eight noise levels, confirming superior robustness over \texttt{MADFO}~\cite{MADFO} at intermediate-to-small budgets.
\subsection{Notation and Organization}

In this paper, we denote $[n]:=\{1,\dots,n\}$, $[n]_0:=\{0,\dots,n\}$, $\mathcal{N}(0,I_n)$ as the standard $n$-dimensional normal distribution, and $a \sim \mathcal{U}(0,1)$ as a uniform random variable on $(0,1)$. For a $\sigma$-algebra $\mathcal{F}$,
$\mathbb{P}(\,\cdot\,\mid\mathcal{F})$ denotes conditional probability given $\mathcal{F}$. We use $\mathbb{E}$, $\|\cdot\|$, $\lceil \cdot \rceil$, $\propto$ and $\emptyset$ to denote the mathematical expectation, Euclidean norm, ceiling function, proportionality and empty set, respectively. The abbreviation
i.i.d. stands for independent and identically distributed random variables. Unless specified as the mathematical constant, $\pi$ represents the permutation induced during the sorting–selection step. We also adopt the conventional asymptotic notation $\mathcal{O}$, $\Omega$,
and $\Theta$ to characterize growth rates, and {$\sup$} for the least upper bound.

The paper is organized as follows. Section~\ref{MethodSec} introduces
\texttt{DAES} and details its mechanism. We establish complexity guarantees
in Sections~\ref{sec:comResults} and~\ref{sec:comResults2} for deterministic
and randomized variants of \texttt{DAES}, respectively. In Section~\ref{NumSec}, numerical experiments are reported. Section~\ref{ConclusionSec} concludes the paper.
\section{{\tt DAES}: A New DFO Method}\label{MethodSec} 
	
This section details the \texttt{DAES} framework for solving \eqref{e.prob}, starting from an initial point $x^0$ at $\ell=0$. 
We detail its three enhanced algorithmic phases and summarize the complete procedure in Subsection~\ref{sec.alg}.

\subsection{Mutation Phase}\label{s.MutPhase}
During iteration $\ell$ of \texttt{DAES}, assuming an even population size $\lambda$, the mutation phase generates three components to construct the offspring population as follows:
\subsubsection{Mutation Directions}\label{s.DisDir}
Although sampling from the standard normal distribution $\mathcal{N}(0,I_n)$ is unbiased in expectation, finite sample sets may exhibit directional asymmetries. Inspired by symmetric search strategies in some DFO methods (e.g., \cite{SDBOX,VRBBO}), \texttt{DAES} mitigates such sample-induced bias by drawing $\lambda/2$ independent random vectors $z^{j\ell}\sim\mathcal{N}(0,I_n)$ and forming symmetric pairs $(+z^{j\ell}, -z^{j\ell})$, for $j\in[\lambda/2]$, thereby guaranteeing a zero empirical mean. Moreover, unlike \texttt{MAES}, which uses a scaling matrix to transform distribution directions into mutation directions, \texttt{DAES} sets this transformation matrix to the identity. This eliminates the computational overhead of matrix adaptations. Thus, mutation directions are derived directly from the sampled directions. To maintain consistent directional scaling, they are normalized and then assembled into
\begin{equation}\label{eq.Dstore}
D^\ell_{\pm}
:=
\left(
\widehat Z^\ell\; -\widehat Z^\ell
\right)
\in\mathbb{R}^{n\times\lambda}, \qquad \text{where} \qquad  \widehat Z^\ell
:=
\left(
\frac{z^{1\ell}}{\|z^{1\ell}\|},
\ldots,
\frac{z^{\lambda/2,\ell}}{\|z^{\lambda/2,\ell}\|}
\right).
\end{equation}
\subsubsection{Mutation Step-Sizes}\label{S.mutStepSize} 
The mutation step-sizes play a
crucial role in ensuring adequate exploration of the search space; see \cite{MADFO,hansen2016cma}. Using the element-wise division $\oslash$, let us define
\begin{equation}\label{eq:bjl-def}
b^{j\ell}:=\left|x_{\trec}^{\ell}\oslash q_\alpha\right|, \qquad q_\alpha\neq 0 \in \mathbb{R}^n.
\end{equation}
Here,
$q_\alpha=z^{j\ell}\sim\mathcal N(0,I_n)$. The division is not well-defined in $\mathbb{R}^n$ whenever $q_\alpha$ has zero entries, making $b^{j\ell}$ undefined. Hence, to obtain the mutation step-sizes $\alpha^{j\ell}$, we first define an active index set $\mathcal I
:=
\left\{
i\in[n]:(q_\alpha)_i\neq0
\right\}$, then evaluate the median of $b^{j\ell}_{\mathcal I}:=(b_i^{j\ell})_{i\in \mathcal I}$ and finally project the resulting value onto $[\alpha_{\min},\alpha_{\max}]$; i.e., 
\begin{equation}\label{e.mutStep}
\alpha^{j\ell}
=
\min\left\{
\alpha_{\max},
\max\left\{
\alpha_{\min},
\sqrt{\sigma^{\ell}\cdot \text{median}(b^{j\ell}_{\mathcal I})}
\right\}
\right\} > 0,
\qquad \text{for}\quad j\in[\lambda/2].
\end{equation}
In \eqref{e.mutStep}, $\sigma^{\ell}$ is the step-size scaling factor  \cite{beyer2017simplify}, updated and projected as
\begin{equation}\label{eq.recomStep0}
{\sigma}^{\ell+1} =
\min\left\{
\alpha_{\max},
\max\left\{
\alpha_{\min},
{\sigma}^{\ell}
\exp\!\left(
{c_{\sigma}}{d_{\sigma}^{-1}}
\left({\|s_\sigma^{\ell}\|_2}/{e_{\sigma}} - 1
\right)
\right)
\right\}
\right\},
\end{equation}
where $\sigma^0\in [\alpha_{\min},\alpha_{\max}]$, $0<c_\sigma\leq 1$ is a learning rate, $d_\sigma>0$ is a damping parameter, typically $d_\sigma\approx 1$, and $e_\sigma>0$ approximates $\mathbb{E}(\|u\|_2)$ for a vector $u\sim\mathcal{N}(0,I_n)$. In~\eqref{eq.recomStep0},
$s_\sigma^{\ell}$ is the {evaluation path},
initialized by $s_\sigma^0=0\in\mathbb{R}^n$ and updated as
\begin{equation}\label{eq:s-sigma-update}
s_\sigma^{\ell+1}
=
(1-c_\sigma)s_\sigma^\ell
+
\sqrt{c_\sigma(2-c_\sigma)}\;\cdot
\frac{
x_{\trec}^{\ell+1}-x_{\trec}^{\ell}
}{
\sigma^\ell
},
\end{equation}
which accumulates past normalized displacements, serving as a filter to smooth out noise and adaptively guide the step-size scaling in \eqref{eq.recomStep0}. 

\subsubsection{Mutation Points}\label{S.mutCandid}
Once the mutation step-size $\alpha^{j\ell}\in [\alpha_{\min},\alpha_{\max}]$ in \eqref{e.mutStep} at iteration $\ell$ is determined, the $\lambda$ mutation points $y^{j\ell}$ can be computed as $y^{j\ell} := (y_+^{j\ell}\;\; y_-^{j\ell})$ defined by
\begin{equation}\label{e.mutPoints}
y^{j\ell}_+ := x_{\trec}^{\ell} + \alpha^{j\ell} {\widehat z^{j\ell}}, \qquad y^{j\ell}_- := x_{\trec}^{\ell} - \alpha^{j\ell} {\widehat z^{j\ell}}, \qquad j\in[\lambda/2],
\end{equation}
where $\widehat z^{j\ell}$ is the $j$-th column of $\widehat Z^\ell$ in \eqref{eq.Dstore}, and $x_{\trec}^{\ell}\in\mathbb{R}^n$ is the current iterate updated in the recombination phase (Subsection~\ref{sec.tri-sp}), with $x_{\trec}^{0}=x^0$. Alternatively, using a unified indexing notation, the $\lambda$ mutation points can be written as
$$y^{j\ell} = x_{\trec}^{\ell} + \alpha^{j\ell} d^{j\ell}, \qquad j\in[\lambda],$$
where $d^{j\ell}$ is the $j$-th column of $D^\ell_{\pm}\in\mathbb{R}^{n\times\lambda}$ in \eqref{eq.Dstore}, with $\alpha^{j+\lambda/2,\ell}=\alpha^{j\ell}$ for all $j\in[\lambda/2]$.
 
\subsection{Sorting--Selection Phase}\label{sec:sort}

To potentially mitigate noise-induced ranking errors that can arise in single-population classical \texttt{MAES} approaches \cite{auger2005restart,beyer2017simplify,MADFO}, \texttt{DAES} is designed to generate $N_{\mathrm{sel}}$ independent populations, retain the most reliable one, and aggregate all its symmetric directions for recombination.

Assuming $\lambda=3\mu$ with $\mu\geq 2$, let $\wt f_\pm^{j\ell} := \wt f(y^{i\ell})$ denote the noisy function evaluations at the $j$-th mutation point $y^{j\ell} := (y_+^{j\ell}\;\; y_-^{j\ell})$ defined in \eqref{e.mutPoints}. We aggregate the evaluated objective values into the vector $\wt F^\ell_{\pm} := ( \wt F^\ell_+ \;\; \wt F^\ell_- ) \in \mathbb{R}^\lambda$
\begin{equation}\label{eq.Fstore}
\wt F^\ell_+ := \left(\wt f_+^{1\ell},\ldots,\wt f_+^{(\lambda/2)\ell}\right)  \quad \text{and} \quad \wt F^\ell_- := \left(\wt f_-^{1\ell},\ldots,\wt f_-^{(\lambda/2)\ell}\right).
\end{equation}

\texttt{DAES} sorts the entries of $\wt F^\ell_{\pm}$ in ascending order,
$\wt f^{1\ell}_\pi \le \cdots \le \wt f^{\lambda\ell}_\pi$, yielding $\wt F^\ell_{\pi}$.
The corresponding sorting permutation is then applied to the columns of $D^\ell_{\pm}$ in \eqref{eq.Dstore} to align each direction with its reordered function value, obtaining $D^\ell_\pi$. Assuming $\lambda = 3\mu$, the columns of $D^\ell_\pi$ are partitioned into three ranked groups with index sets:
$$ \mathcal{G}_\pi^1 := \{1, \dots, \mu\}, \qquad \mathcal{G}_\pi^2 := \{\mu+1, \dots, 2\mu\}, \qquad \mathcal{G}_\pi^3 := \{2\mu+1, \dots, 3\mu\}, $$
corresponding to the $\mu$ best, intermediate, and worst objective values, respectively.

Let us define the best element associated with the $q$-th population and $k$-th group as
\begin{equation}\label{eq.jkq}
{j_k^q
\in
\argminx_{j\in\mathcal G_\pi^k}
\wt f_\pi^{j\ell},
\qquad
k\in[3], \qquad q\in[N_{\mathrm{sel}}].}
\end{equation}
Then, the
\textit{\textbf{minimum within-group noisy separation}} is defined by
\begin{equation}\label{eq:selection-separation}
\Delta_{\mathrm{sel}}^{\ell,q}
:=
\min_{k\in[3]}
\min_{r\in\mathcal G^k_\pi\setminus\{j_k^q\}}
\left\{
\wt f_\pi^{r\ell}
-
\wt f_\pi^{j_k^q\ell}
\right\}, \qquad q\in[N_{\mathrm{sel}}]. 
\end{equation}
After
generating $N_{\mathrm{sel}}$ independent populations and observing $\Delta_{\mathrm{sel}}^{\ell,q}$ for $q\in[N_{\mathrm{sel}}]$, {\tt DAES} selects
the population having the largest observed within-group separation; i.e.,
\begin{equation}\label{eq:qsel}
{q^\star\in
\argmaxx_{q\in[N_{\mathrm{sel}}]}
\Delta_{\mathrm{sel}}^{\ell,q}.}
\end{equation}
This is intended to enhance ranking (sorting) reliability under unknown noise $\omega$ and mitigate potential rank inversion. Moreover, this criterion plays a key analytical role in Section~\ref{sec:comResults2}.
\begin{algorithm}[H]
\setstretch{1.12}
\caption{Mutation and sorting--selection procedure in {\tt DAES}}
\label{alg.daes-selection}
\begin{algorithmic}[1]

\State \textbf{Input:}
$x^\ell_{\trec}$,
$\sigma^\ell$, $s_\sigma^{\ell}$,
$N_{\mathrm{sel}}$,
and $\lambda = 3\mu$ where $\mu\geq 2$;
\For{$q\in[N_{\mathrm{sel}}]$}
    \Comment{\textit{Independent populations}}

    \For{$j\in[\lambda/2]$}
        \State \btr sample
        $z^{j\ell}\sim\mathcal{N}(0,I_n)$;
        \State \btr obtain
        $\alpha^{j\ell}$
        via \eqref{e.mutStep}
        with $q_\alpha:=z^{j\ell}$, and set $\alpha^{j+\lambda/2,\ell}:=\alpha^{j\ell}$; 
         \State \btr generate $y^{j\ell}_{\pm}$ via
        \eqref{e.mutPoints} and evaluate
        {$\wt f(y^{j\ell}_{\pm})$};
    \EndFor
    \State \btr form $D^{\ell}_{\pm}$ via \eqref{eq.Dstore}, $\wt F^\ell_{\pm} := \big( \wt F^\ell_+ \;\; \wt F^\ell_- \big)$ via \eqref{eq.Fstore} and $\wt F^\ell_{\pi}$ after ascending sorting;
    \State \btr permute $D^{\ell}_{\pm}$ to get $D^\ell_\pi$ and partitioned groups $\mathcal G_\pi^k$, $k\in[3]$;    
    \State \btr compute $\Delta_{\mathrm{sel}}^{\ell,q}$ via
    \eqref{eq:selection-separation};
\EndFor
\State \btr set
$q^\star =
\argmaxx_{q\in[N_{\mathrm{sel}}]}
\Delta_{\mathrm{sel}}^{\ell,q}$ as in \eqref{eq:qsel};
\State \textbf{Output:} $D^{\ell}_{\pi}$ and $\mathcal G_\pi^k$, $k\in[3]$, associated with $q^\star$; 
\end{algorithmic}
\end{algorithm}
\subsection{Recombination Phase}\label{sec:recomPhase}
This phase converts the ranked candidates retained from the sorting--selection phase into a robust search direction $p_{\mathrm{trec}}^\ell$. Subsection~\ref{seb.ptrec} forms group-representative vectors via within-group weighting and combines them through a triangular construction to obtain $p_{\mathrm{trec}}^\ell$. Subsection~\ref{sec.tri-sp} applies this direction in a noisy descent condition, triggers extrapolation, and performs success–failure updates on the iterate.
\subsubsection{Triangular Recombination Direction}\label{seb.ptrec}  
Using $D^\ell_\pi$,  \texttt{DAES} constructs group-representative recombination directions, for $k\in[3]$,
\begin{equation}\label{e.RecDirs}
d^\ell_{\mathrm{rec}\,k}
=
\sum_{j\in\mathcal G_\pi^k} 
\bar w_j d^{j\ell}_\pi,
\qquad \text{such that}\; \sum_{j\in\mathcal G_\pi^k}\bar w_j=1.
\end{equation}
Here, $\bar{w}_j > 0$ is the \textbf{\textit{normalized within-group weight}} assigned to the $j$-th direction. These weights are defined below, with rationale provided in Subsection~\ref{sec:randomDirections}.

Let $t_a>0$ be a \textbf{\textit{baseline alignment threshold}} from high-dimensional isotropic Gaussian sampling. Under noisy evaluations and adaptive step-sizes in $[\alpha_{\min},\alpha_{\max}]$, for the best direction indexed as ${j_k^\star \in
\argminx_{j\in\mathcal G_\pi^k}
\wt f_\pi^{j\ell}}$ in $\mathcal G_\pi^k$, we define a \textbf{\textit{particular alignment threshold}} $t_{\mathrm{sel}}>0$. Obviously, $j_1^\star = 1$, $j_2^\star = \mu+1$ and $j_3^\star = 2\mu + 1$. To ensure that the other elements within $\mathcal{G}_\pi^k$ do not cancel the contribution of the best direction, the \textbf{\textit{within-group dominance parameter}} $\varepsilon_{\bar{w}}\in(0,1)$ is defined such that 
\begin{equation}\label{eq:barepsilon-def} 
\varepsilon_{\bar{w}} < \frac{t_{\mathrm{sel}}}{1 + t_{\mathrm{sel}}}, \qquad \text{where} \qquad t_{\mathrm{sel}}:=({\alpha_{\min}}/{2\alpha_{\max}})t_a, \quad t_a:={\tau}/{\sqrt{n}}, 
\end{equation}
with $n\geq 2$ and $0<\tau<\sqrt{\pi/2}$. Given $\varepsilon_{\bar w}\in(0,1)$, define 
\begin{equation}\label{eq:wd-algorithmic-weights}
\bar w_{j_k^\star}
:=
1-\varepsilon_{\bar w},
\qquad
\bar w_j
:=
\frac{\varepsilon_{\bar w}}{\mu-1},
\qquad
j\in\mathcal G_\pi^k\setminus\{j_k^\star\},
\qquad
k\in[3].
\end{equation}

We now formulate the main search direction initiated from the following points:
\begin{equation}\label{e.RecPoints}  
x^{\ell}_{\rec k} := x_{\trec}^{\ell} + \alpha^\ell_{\rec k} d^\ell_{\rec k}, \qquad \alpha^\ell_{\text{rec}\,k}
\in [\alpha_{\min}, \alpha_{\max}], 
\qquad
k\in[3],
\end{equation}
where $d^\ell_{\mathrm{rec}\,k}$ is a direction given in \eqref{e.RecDirs}, yielding a well-defined step-size $\alpha^\ell_{\text{rec}\,k}$ via \eqref{e.mutStep}, with $q_\alpha:=d^\ell_{\mathrm{rec}\,k}$.
Note that any $d^{j\ell}_\pi$ is unit based on \eqref{eq.Dstore}, thus
\begin{equation}\label{eq.posqa}
    \|d^\ell_{\mathrm{rec}\,k}\| \ge \bar w_{j_k^\star} \|d^{j_k^\star\ell}_\pi\| - \sum_{j \neq j_k^\star} \bar w_j \|d^{j\ell}_\pi\|  = 1 - 2\varepsilon_{\bar w} > 0.
\end{equation}

Given \eqref{e.RecPoints}, as shown in Fig.~\ref{fig:triangle}, let us consider the main triangle $\triangle(v_1^\ell, v_2^\ell, v_3^\ell)$ formed by the vertices $v_k^\ell := x_{\mathrm{rec}\,k}^\ell$, for $k \in [3]$. Near vertices $v_1^\ell$ and $v_2^\ell$, we also define the sub-triangles $\triangle(v_1^\ell, v_{12}^\ell, v_{13}^\ell)$ and $\triangle(v_2^\ell, v_{21}^\ell, v_{23}^\ell)$, respectively, where the midpoints are
\begin{equation}\label{e.vv}
v_{12}^\ell = v_{21}^\ell
:=
\frac{v_1^\ell + v_2^\ell}{2},
\qquad
v_{13}^\ell
:=
\frac{v_1^\ell + v_3^\ell}{2},
\qquad
\text{and}
\qquad
v_{23}^\ell
:=
\frac{v_2^\ell + v_3^\ell}{2}.
\end{equation}
\begin{figure}[t]
\centering
\includegraphics[width=0.42\linewidth]{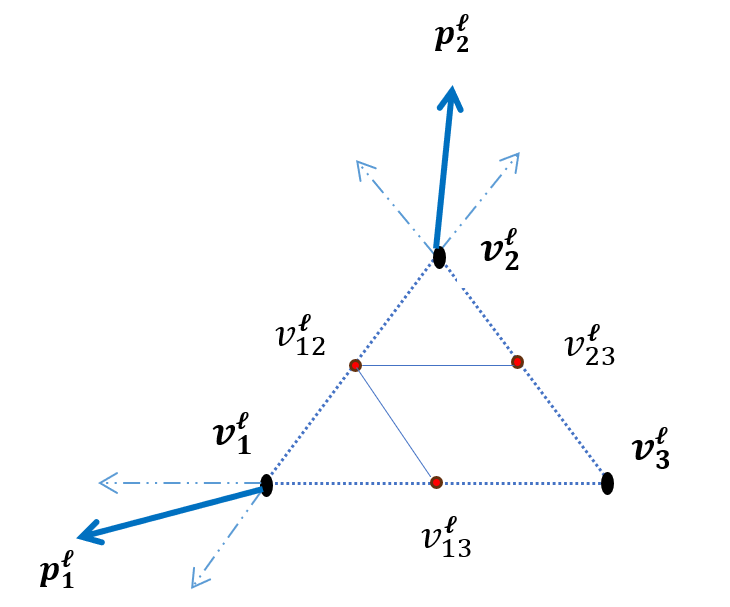}
\caption{Generating two components for the triangular recombination}
\label{fig:triangle}
\end{figure}
Letting $\Phi(\eta,a,b):=\eta a+(1-\eta)b$, we introduce a preliminary direction as
\begin{equation}\label{e.HeuDir_p}
\wt p_{\trec}^\ell
=
\frac{1}{w_p^\ell}\Phi(\eta,p^{\ell}_{1},p^{\ell}_{2})-
x_{\trec}^{\ell},
\qquad
\eta \in (0,1),
\end{equation}
where 
\[p^{\ell}_{1} := h^\ell v_{12}^\ell + b^\ell v_{13}^\ell, \qquad
p^{\ell}_{2} := \bar{h}^\ell v_{21}^\ell + \bar{b}^\ell v_{23}^\ell, \qquad
w_p^\ell:=\Phi(\eta,h^\ell+b^\ell,\bar{h}^\ell + \bar{b}^\ell),
\]
such that
\begin{equation}\label{eq:hhbb}
    h^\ell, \bar{h}^\ell \sim \mathcal{U}(0,1), \quad
b^\ell := \sqrt{1-(h^\ell)^2}, \quad
\bar{b}^\ell := \sqrt{1-(\bar{h}^\ell)^2}.
\end{equation}
Rather than imposing a deterministic combination of vertices, $\wt p_{\trec}^\ell$ samples a trial position inside $\triangle(v_1^\ell, v_2^\ell, v_3^\ell)$, biasing the search trajectory toward the superior vertices $v_1^\ell$ and $v_2^\ell$. Subtracting $x_{\trec}^\ell$ is mathematically necessary to convert the trial point into a centered search direction, eliminating positional bias. Algebraically, $\wt p_{\trec}^\ell$ simplifies to a weighted combination of $v_1^\ell, v_2^\ell, v_3^\ell$, where the emerging coefficients correspond to:
\begin{equation}\label{e.ttt}
\begin{aligned}
\wt\theta^\ell_1
:=
\frac{1}{2}\Phi(\eta,h^\ell + b^\ell,\bar{h}^\ell),\qquad
\wt\theta^\ell_2
&:=
\frac{1}{2}\Phi(\eta,h^\ell,\bar{h}^\ell + \bar{b}^\ell),
\qquad
\wt\theta^\ell_3
:=
\frac{1}{2}\Phi(\eta,b^\ell,\bar{b}^\ell).
\end{aligned}
\end{equation}
It is easy to verify that
\begin{equation}
\sum_{k=1}^{3} \wt\theta_k^\ell = w_p^\ell, \qquad w_p^\ell>0.
\label{e.w}
\end{equation}

However, contributions from $v_2^\ell$ and $v_3^\ell$ may disrupt the trajectory of $v_1^\ell$ as the superior vertex--associated with the superior group ($\mathcal G_\pi^1$)--because raw weights \eqref{e.ttt} depend on local random geometry. To maintain the directional integrity of $v_1^\ell$ (aligned with $d_{\pi}^{1\ell}$), we define the \textbf{cross-group scaling factor}
\begin{equation}\label{eq:theta-margin-factor}
s_\theta^\ell := \min \left\{ 1, \; \frac{\rho \bar{\varepsilon} \, \wt{\theta}_1^\ell \alpha_{\mathrm{rec},1}^\ell}{2 \sum_{k \neq 1} \wt{\theta}_k^\ell \alpha_{\mathrm{rec},k}^\ell} \right\}, \qquad \rho \in (0,1)
\end{equation}
and the \textbf{triangular weights} as
\begin{equation}\label{eq:theta-margin-weights}
\theta_{1}^\ell
:=
\wt\theta_{1}^\ell,
\qquad
\theta_2^\ell
:=
s_\theta^\ell\wt\theta_2^\ell,
\qquad
\theta_3^\ell
:=
s_\theta^\ell\wt\theta_3^\ell.
\end{equation}
Combining \eqref{eq:theta-margin-factor} and \eqref{eq:theta-margin-weights} directly yields $\sum_{k\neq 1} \theta_k^\ell\alpha_{\mathrm{rec},k}^\ell \le \frac{1}{2} \rho\bar\varepsilon\, \theta_{1}^\ell \alpha_{\mathrm{rec},1}^\ell$, guaranteeing the intrinsic \textbf{cross-group margin property} of \texttt{DAES}:
\begin{equation}\label{e.marginCon}
\theta_{1}^\ell \alpha_{\rec,1}^\ell \bar\varepsilon \rho - \sum_{k\neq 1} \theta_k^\ell\alpha_{\rec k}^\ell > 0.
\end{equation}
Using \eqref{eq:theta-margin-weights}, the \textbf{unscaled triangular recombination direction} is
\begin{equation}\label{e.trid}
\widehat p_{\trec}^\ell
=
\dfrac{
\sum_{k=1}^3
\theta_k^\ell\,
\alpha^\ell_{\rec k}\,
d^\ell_{\rec k}
}{\sum_{k=1}^{3} \theta^\ell_k}.
\end{equation}
Finally, normalizing $\widehat p_{\trec}^\ell$ yields the \textbf{triangular recombination direction} as
\begin{equation}\label{eq:ptrec-normalized}
p_{\trec}^\ell
=
\delta
\frac{\widehat p_{\trec}^\ell}
{\|\widehat p_{\trec}^\ell\|},
\qquad
\widehat p_{\trec}^\ell\neq0, \qquad \text{for some constant} \quad \delta>0.
\end{equation}
If, exceptionally,
$\widehat p_{\trec}^\ell=0$, setting
$p_{\trec}^\ell:=\delta d_\pi^{1\ell}$ preserves
$\|p_{\trec}^\ell\|=\delta$ as
$\|d_\pi^{1\ell}\|=1$.

\begin{algorithm}[H]
\setstretch{1.1}
\caption{Triangular recombination direction in {\tt DAES}}
\label{alg.daes-recombination}
\begin{algorithmic}[1]

\State \textbf{Input:}
$\sigma^\ell$, $s_\sigma^{\ell}$, $x^\ell_{\trec}$,
$\wt f^\ell_{\trec}$,
the ranked matrix $D^{\ell}_{\pi}$ and groups
    $\mathcal G_\pi^k$, $k\in[3]$;

\State \btr generate the within-group weights
$\bar w_j$ via \eqref{eq:wd-algorithmic-weights};

\State \btr compute $d^\ell_{\rec k}$ via \eqref{e.RecDirs},
for $k\in[3]$;

\State \btr obtain
$\alpha^\ell_{\rec k}$
via \eqref{e.mutStep}
with $q_\alpha:=d^\ell_{\rec k}$,
for $k\in[3]$;

\State \btr compute $x^\ell_{\rec k}$ via \eqref{e.RecPoints},
for $k\in[3]$;

\State \btr sample
$h^\ell,\bar h^\ell\sim\mathcal U(0,1)$ and obtain
$b^\ell$ and $\bar b^\ell$ via \eqref{eq:hhbb};

\State \btr generate $\theta_k^\ell$ via \eqref{eq:theta-margin-weights}, for $k\in[3]$;

\State \btr construct $\widehat p_{\trec}^\ell$ via \eqref{e.trid};

\State \btr  {\bf if} $\widehat p_{\trec}^\ell\neq0$ {\bf then}, obtain $p_{\trec}^\ell$ via
    \eqref{eq:ptrec-normalized}; {\bf else}, set
    $p_{\trec}^\ell:=\delta d_\pi^{1\ell}$; {\bf end if}
\State \textbf{Output:}
$p^\ell_{\trec}$;
\end{algorithmic}
\end{algorithm}
\begin{algorithm}[t]
\setstretch{1.1}
\caption{Triangular recombination point in {\tt DAES}}
\label{alg.daes-extrapolation}
\begin{algorithmic}[1]
\State \textbf{Input:}
$\alpha^{\ell0}$,
$p^\ell_{\trec}$,
$x^\ell_{\trec}$ and
$\wt f^\ell_{\trec}$;
\State \btr set $\mathsf{succ}_\ell:=0$;
\For{$s\in\{+1,-1\}$}
    \Comment{\textit{Search along signed directions}}
    \State \btr evaluate
    $\wt f_{\text{val}}
    :={
    \wt f(x^\ell_{\trec}+s\alpha^{\ell0}p^\ell_{\trec})}$;
\State \textbf{if} $\wt\mu_s(\alpha^{\ell0}) > \wt\beta$ \textbf{then} set $\alpha^\ell_{\trec}:=\alpha^{\ell0}$, 
 $\mathsf{succ}_\ell:=1$, and \textbf{break}; \textbf{end if}
\EndFor
\If{$\mathsf{succ}_\ell$}
    \Comment{\textit{Extrapolation along selected direction}}
    \While{$1$}
        \State \btr set
        $\alpha_{\mathrm{temp}}:=\alpha^\ell_{\trec}$
        and
        $\wt f_{\mathrm{temp}}:=\wt f_{\text{val}}$;

        \State \btr set
        $\alpha^\ell_{\trec}:=\gamma_e\alpha^\ell_{\trec}$;
        evaluate
        $\wt f_{\text{val}}
        :={
        \wt f(x^\ell_{\trec}
        +s\alpha^\ell_{\trec}p^\ell_{\trec})}$;

        \If{$\wt\mu_s(\alpha^\ell_{\trec}) \leq \wt\beta$}
            \State \btr set
            $\alpha^\ell_{\trec}:=\alpha_{\mathrm{temp}}$
            and
            $\wt f_{\text{val}}:=\wt f_{\mathrm{temp}}$;
            \textbf{break};
        \EndIf
    \EndWhile
    \State \btr set
    $x^{\ell+1}_{\trec}
    :=x^\ell_{\trec}
    +s\alpha^\ell_{\trec}p^\ell_{\trec}$
    and
    $\wt f^{\ell+1}_{\trec}:=\wt f_{\text{val}}$;
\Else
    \State \btr set
    $x^{\ell+1}_{\trec}:=x^\ell_{\trec}$
    and
    $\wt f^{\ell+1}_{\trec}:=\wt f^\ell_{\trec}$;
\EndIf
\State \textbf{Output:}
$x^{\ell+1}_{\trec}$ and $\mathsf{succ}_\ell$;
\end{algorithmic}
\end{algorithm}
\subsubsection{Triangular Recombination Point}\label{sec.tri-sp}\label{sec.initial-search-scale}
Given $p_{\trec}^\ell$, the iterate $x_{\trec}^\ell$ is robustly updated via the following steps:
\begin{enumerate}
    \item \textbf{\textit{Initial signed-search step-size.}} It is computed via \eqref{e.mutStep} using $q_\alpha := p_{\trec}^\ell$ as a candidate $\hat{\alpha}^{\ell0}$, and safeguarded within $[L_\ell, a^\ell]$, i.e.,
    \begin{equation}\label{e.alpha0}
    \alpha^{\ell0} := \min \left\{ a^\ell, \, \max \left\{ L_\ell, \, \hat{\alpha}^{\ell0} \right\} \right\}, \quad \text{where } L_\ell := \max \{ \alpha_{\min}, \, c_\alpha^{-1} a^\ell \},
    \end{equation}
    and $c_\alpha>1$ is a constant and $a^\ell$ is an initial auxiliary step-size scaling factor, satisfying
    \begin{equation}\label{eq:initial-search-range}
    {\max\left\{\alpha_{\min},c_\alpha^{-1}a^\ell\right\}\leq\alpha^{\ell0}\leq a^\ell \leq \alpha_{\max}.} 
    \end{equation}
    \item \textit{\textbf{Signed-search.}} Given $\alpha:=\alpha^{\ell0}$, the candidate $x_{\trec}^\ell + s \alpha p_{\trec}^\ell$ is evaluated using the \textbf{noisy derivative-free sufficient descent condition} ({\tt NDF-SDC}):
    \begin{equation}\label{e.SDC-nograd}  
    \wt\mu_s(\alpha) > \wt\beta, \qquad \text{for some} \quad \wt\beta\in(0,1) \;\; \text{and} \;\; s\in\{-1,1\},
    \end{equation}
    where
    \begin{equation}\label{e.muW-nograd}
    \wt\mu_s(\alpha)
    :=
    \frac{\wt f(x^{\ell}_{\trec}+ s\alpha p^{\ell}_{\trec})
    - \wt f(x^{\ell}_{\trec})}{-\alpha^2}, \qquad  \alpha > 0.
    \end{equation}
      
    \item \textit{\textbf{Extrapolation and scale refinements.}}
    The candidate (and thus iteration $\ell$) is classified as {successful} if \eqref{e.SDC-nograd} holds for at least one $s \in \{+1, -1\}$; otherwise, it is deemed {unsuccessful}. Given a {successful direction}, an extrapolation expansion successively inflates the step-size $\alpha := \gamma_e \cdot \alpha$ until \eqref{e.SDC-nograd} fails, where $\gamma_e > 1$. Using the last successful step-size $\alpha_{\trec}^\ell:=\alpha$ as the \textbf{\textit{triangular recombination step-size}}, the iterate is updated to the
    \textit{\textbf{triangular recombination point}}
    \begin{equation}\label{e.RecPoint}
    x_{\trec}^{\ell+1}
    = x_{\trec}^{\ell}\pm\alpha^{\ell}_{\trec} p_{\trec}^{\ell}, \qquad x_{\trec}^{0}=x^0.
    \end{equation}
    Given an {unsuccessful direction}, the iterate remains unchanged, while the step-size scaling factors are contracted by $\rho_u \in (0, 1)$ as:
    \begin{equation}\label{eq:unsuccessful-sigma-contraction}
    a^{\ell+1} := \max\{\alpha_{\min}, \rho_u a^\ell\}, \qquad \sigma^{\ell+1} := \max\{\alpha_{\min}, \rho_u \sigma^\ell\}.
    \end{equation} 
\end{enumerate}

Note that $\sigma^\ell \in [\alpha_{\min}, \alpha_{\max}]$ scales mutation/recombination step-sizes (including initial signed ones), while $a^\ell \in [\alpha_{\min}, \alpha_{\max}]$ is an auxiliary factor safeguarding only the initial signed step-size. In successful iterations, $\sigma^\ell$ updates via \eqref{eq.recomStep0}, whereas $a^\ell$ remains unchanged. The refinements \eqref{eq:unsuccessful-sigma-contraction} are crucial for convergence in unsuccessful iterations. 

\subsection{The {\tt DAES} Algorithm}\label{sec.alg}

To complete the algorithmic specification of \texttt{DAES}, we detail the explicit calibration of the population size $\lambda$ and outline the complete procedure.

\textbf{\textit{Population size $\mathbf{\lambda}$. }} Let $\ell_{\max}$ be the prescribed maximum outer iteration limit, distinct from the stationarity index $\ell_T \le \ell_{\max}$ used in the complexity analysis. Given an \textit{\textbf{overall confidence parameter}} $\gamma_0 \in (0, 1/2)$, the \textit{\textbf{per-population confidence budget}} $\gamma_{\mathrm{pop}}$ is defined, from which the population size $\lambda$ is explicitly calibrated as: 
\begin{equation}\label{eq:lambda-confidence}
\lambda = 6 \left\lceil \frac{\log(\gamma_{\mathrm{pop}}^{-1})}{3\log(1/(\tau\sqrt{2/\pi}))} \right\rceil,
\qquad
\text{where}
\ \ 
\gamma_{\mathrm{pop}} := \frac{\gamma_0}{\ell_{\max}N_{\mathrm{sel}}}, \ \ 0<\tau<\sqrt{\pi/2}.
\end{equation}
The multiplier $6$ allows the $\lambda/2$ symmetric direction pairs to be partitioned into three ranked groups. In the randomized analysis (Section~\ref{sec:comResults2}), this logarithmic calibration ensures that the sampling-failure probability of each population remains bounded by $\gamma_{\mathrm{pop}}$, securing the high-probability convergence guarantees.

\textbf{\textit{Algorithmic Outline. }} Algorithm~\ref{alg.daes} integrates three subroutines into the complete pseudo-code for a single outer iteration. First, it generates some populations and executes the sorting--selection mechanism (Algorithm~\ref{alg.daes-selection}) to isolate the most noise-resilient population. Next, the triangular recombination procedure (Algorithm~\ref{alg.daes-recombination}) synthesizes the search direction from that population. Finally, the signed-search and extrapolation scheme determines the step-size and updates the iterate  (Algorithm~\ref{alg.daes-extrapolation}). 

\begin{algorithm}[t]
\setstretch{1.1}
\caption{{\tt \textbf{DAES}}: the $(\ell+1)$-th iteration}
\label{alg.daes}
\begin{algorithmic}[1]

\State \textbf{Tuning parameters:}
$\ell_{\max},N_{\mathrm{sel}}\in\mathbb N$, $\gamma_0\in(0,1/2)$,
$0<\tau<\sqrt{\pi/2}$,
$\wt\beta\in(0,1)$,
$\eta, \rho\in(0,1)$, $\delta>0$, $\gamma_e>1$, weight budget $\varepsilon_{\bar w}$ in \eqref{eq:barepsilon-def}, parameters required in \eqref{eq.recomStep0},
$\rho_u\in(0,1)$,
$c_\alpha>1$,
$0<\alpha_{\min}<\alpha_{\max}<\infty$,
and $a^0\in[\alpha_{\min},\alpha_{\max}]$;

\State \textbf{Input:}
$\sigma^{\ell}$, $s_\sigma^\ell$,
$a^\ell$,
$x_{\trec}^{\ell}$, $\mathsf{succ}_\ell$, and
$\wt f_{\trec}^{\ell}:=\wt f(x_{\trec}^{\ell})$;

\State set the population size $\lambda$ according to
\eqref{eq:lambda-confidence};

\Statex \hrulefill

\State \btr apply Algorithm~\ref{alg.daes-selection} to have $D^{\ell}_{\pi}$ and 
    $\mathcal G_\pi^k$, $k\in[3]$, associated with $q^\star$; 

\State \btr apply Algorithm~\ref{alg.daes-recombination} to have $p^\ell_{\trec}$;

\State \btr compute 
$\hat{\alpha}^{\ell0}$ via \eqref{e.mutStep}
with $q_\alpha:=p^\ell_{\trec}$ and obtain  $\alpha^{\ell0}$ via \eqref{e.alpha0};

\State \btr apply Algorithm~\ref{alg.daes-extrapolation} to have
$\mathsf{succ}_\ell$ and $x^{\ell+1}_{\trec}$;

\State \btr update $s_\sigma^{\ell+1}$ via \eqref{eq:s-sigma-update};
\If{$\mathsf{succ}_\ell$}
\State \btr update $\sigma^{\ell+1}$ via \eqref{eq.recomStep0},
set $a^{\ell+1}:=a^\ell$;
\Else
\State \btr reduce $\sigma^{\ell+1}$ and $a^{\ell+1}$  via
\eqref{eq:unsuccessful-sigma-contraction};
\EndIf
\State \textbf{Output:}
$x_{\trec}^{\ell+1}$,
$\wt f_{\trec}^{\ell+1}$,
$s_\sigma^{\ell+1}$,
$\sigma^{\ell+1}$,
and $a^{\ell+1}$;
\end{algorithmic}
\end{algorithm} 
\section{Analytical Study of \texttt{DAES}}
Adopting standard assumptions, this section establishes deterministic (Subsection~\ref{sec:comResults}) and high-probability (Subsection~\ref{sec:comResults2}) complexity bounds for \texttt{DAES}. To our knowledge, this is the first complexity analysis for a noisy \texttt{MAES}-type method across nonconvex, convex, and strongly convex objectives, matching the rates in \cite{VRBBO,VRDFON,bergou2020stochastic,Gratton2015}.
\subsection{Assumptions}
\begin{assumption}
	$\AS$. Given the starting point $x^0\in \mathbb{R}^n$, the following level set is compact:  
	\begin{equation}\label{e.levelSet}
		\mathcal{L}(x^0):=\{x\in \mathbb{R}^n \mid f(x)\leq f(x^0)\}.
	\end{equation}
\end{assumption}
\begin{assumption}
$\AL$. The objective function $f(x)$ is continuously differentiable on $\mathbb{R}^n$, and its gradient $g(x)$ is Lipschitz continuous; i.e., there exists a constant $L > 0$ such that
	\begin{equation}\label{Li}
		\|g(x)-g(y)\| \leq L~\|x-y\|, \qquad  \text{for all $x,y \in \mathbb{R}^n$}.
	\end{equation}
\end{assumption}
These assumptions ensure that $f$ attains its minimum $\underline{f} := f(\underline{x}) > -\infty$ on $\mathcal{L}(x^0)$ and satisfies the Taylor-error bound
\begin{equation}\label{e.TayErr1}
	|f(x+\hat s)-f(x)-\hat s^Tg(x)| \leq \frac{L}{2}|\hat s|^2, \qquad  \text{for all $x,\hat s \in \mathbb{R}^n$}.
\end{equation}

\begin{assumption}
	$\AN$. For some noise level
	\begin{equation}\label{eq:noise-upper-bound}
		0<\omega<\frac{\wt\beta}{2}\alpha_{\min}^2,
		\qquad
		\wt\beta\in(0,1),
	\end{equation}
	the noisy function $\wt f(x)$ satisfies
	\begin{equation}\label{e.noisecon}
		|\wt f(x)-f(x)|
		\leq
		\omega,
		\qquad
		\text{for all }x\in\mathbb{R}^n.
	\end{equation}
\end{assumption}

Condition \eqref{eq:noise-upper-bound} is purely analytical and not computed by \texttt{DAES}. Using the aforementioned assumptions, for the theoretical analysis, we also define two problem-dependent constants 
\begin{equation}\label{eq:Amax-bound}
\beta := \wt\beta - \frac{2\omega}{\alpha_{\min}^2} \in (0,1), 
\qquad \text{and} \qquad 
A_{\max} := \sqrt{\beta^{-1}\big(f(x^0) - \underline{f}\big)}.
\end{equation}
Intuitively, $\beta$ is the effective decrease factor on the true objective $f$, while $A_{\max}$ bounds the length of any successful step under this decrease.

\textbf{\textit{Approximate Stationarity.}} Solving \eqref{e.prob} generally yields at best an $\varepsilon_\omega$-approximate stationary point for some limiting accuracy $\varepsilon_\omega>0$. Accordingly, algorithm complexity is measured by the function evaluations required to find $x\in\Lc(x^0)$ satisfying 
\begin{equation}\label{eq:approx_stationarity}
\|g(x)\|
\le
\varepsilon_\omega,
\qquad\text{and}\qquad
f(x)
\le
\sup\bigl\{
f(y)\mid
y\in\Lc(x^0),\;
\|g(y)\|\le\varepsilon_\omega
\bigr\}.
\end{equation}
While mathematically redundant, this second requirement provides a meaningful worst-case objective-value benchmark.

\subsection{Complexity with Deterministic Directions}\label{sec:comResults}

The deterministic results proceed as follows, laying the foundation for the randomized analysis in Subsection~\ref{sec:comResults2}. Step-sizes satisfying {\tt NDF-SDC} yield true sufficient decrease and finite extrapolation termination (Lemma~\ref{lem:noisy-to-noiseless}, Proposition~\ref{prop:termination}), while directional-derivative bounds (Proposition~\ref{prop:grad-inner-sdc}) yield local gradient bounds (Theorem~\ref{thm:gradnorm-bound}). Crucially, Lemma~\ref{lem:scale-crossing} proves that successive unsuccessful iterations drive the search scale into the noise-balanced regime, establishing the worst-case bound (Corollary~\ref{coll}) and yielding a global gradient bound (Theorem~\ref{thm:gradnorm-bound2}). Next, Theorem~\ref{t.c-noise1} bounds optimality gaps and iterate errors, establishing sublinear (convex) and linear (strongly convex) accuracy. Finally, Proposition~\ref{lem:noisy-eff-academic} bounds decrease efficiency, enabling Theorem~\ref{t.comf-noise1} to establish evaluation complexity across nonconvex, convex, and strongly convex functions. All proofs are provided in Appendix~\ref{app:deterministic-proofs}.

\begin{remark}\label{rem:asymptotic-convention}
Although fixed, the alignment threshold $\Delta_{\angle}$ and the Lipschitz constant $L > 0$ are explicitly retained in Theorems~\ref{thm:gradnorm-bound2}--\ref{t.comf-noise1} to isolate the distinct contributions of directional geometry and objective-governed accuracy, while facilitating a direct comparison with other methods in Table~\ref{tab:randomized-complexity-comparison}.
\end{remark}

Let us define
\begin{equation}\label{eq:wxp}
v(x^\ell,p^\ell)
:=
\begin{cases}
\D\frac{\|g(x^\ell)\|\|p^\ell\|}{|g(x^\ell)^Tp^\ell|},
& \text{if }|g(x^\ell)^Tp^\ell|>0,$$\\
+\infty,
& \text{if }g(x^\ell)\ne0 \; \;  \text{and} \; \; g(x^\ell)^Tp^\ell=0,$$\\
1,
& \text{if }g(x^\ell)=0,
\end{cases}
\end{equation}
to quantify the geometric effect of the alignment (angle) condition 
\begin{equation}\label{eq:angle_condition}
|g(x^\ell)^T p^\ell| \ge \Delta_{\angle}\,\|g(x^\ell)\|\,\|p^\ell\|,
\qquad
\text{for some}
\quad
\Delta_{\angle}\in(0,1].
\end{equation}
\begin{lemma}[\textbf{Consistency of {\tt NDF-SDC}}]\label{lem:noisy-to-noiseless}
Under Assumption~$\ANn$, let $p^\ell\in\mathbb{R}^n$ be a deterministic search
direction in {\tt DAES} and let $\alpha\in[\alpha_{\min}, \infty)$ be a corresponding
step-size at iteration $\ell$ which satisfies {\tt NDF-SDC} \eqref{e.SDC-nograd} for some $s\in\{-1,1\}$. Then, for the constant $\beta\in(0,1)$ defined in
\eqref{eq:Amax-bound}, the
corresponding noiseless sufficient decrease condition is guaranteed to hold,
i.e.,
\begin{equation}\label{eq:noiseless-satisfied}
    \mu_s(\alpha) > \beta, \qquad \text{where}\quad
    \mu_s(\alpha)
    :=
    \frac{f(x^\ell+s\alpha p^\ell)-f(x^\ell)}{-\alpha^2}.
\end{equation}

\textbf{Proof:} See Appendix~\ref{pLemma1}.
\end{lemma}
\begin{proposition}[\textbf{Termination of extrapolation}] 
\label{prop:termination}
Under Assumptions~$\ASn$, $\ALn$, and $\ANn$, let $\alpha_{\min}$ be the lower step-size limit, while $A_{\max}$ is the bound defined in \eqref{eq:Amax-bound}. If $A_{\max} > \alpha_{\min}$, then the total number of expansion trials in \texttt{DAES} (including the final unsuccessful one) is at most
\begin{equation}\label{eq:Nexpand}
N_{e}
:=
\left\lceil
\log_{\gamma_e}
\left({A_{\max}}/{\alpha_{\min}}
\right)
\right\rceil, \qquad \gamma_e > 1.
\end{equation}
If $A_{\max} \le \alpha_{\min}$, no successful outer iteration initiates extrapolation trials; i.e., $N_{e}=0$.

\textbf{Proof:} See Appendix~\ref{pPro1}.
\end{proposition}

\begin{proposition}[\textbf{Directional derivative bound}]
\label{prop:grad-inner-sdc} 
Under Assumptions~$\ASn$, $\ALn$, and $\ANn$, and the analytical compatibility condition \eqref{eq:Amax-bound}, let
$p^\ell\in\mathbb{R}^n$ be a deterministic search direction in
{\tt DAES} with $\delta=\|p^\ell\|$. Then, for both unsuccessful and successful iterations,
\begin{equation}\label{eq:upBound}
|g(x^\ell)^T p^\ell|
\le
\gamma_e\alpha
\left(\wt\beta+\frac{L}{2}\delta^2\right)
+
\frac{2\omega}{\alpha}, \qquad \text{for some constants}\quad \wt\beta\in(0,1),\; \gamma_e>1,
\end{equation}
where $\alpha := \alpha^{\ell0}$ for unsuccessful iterations and $\alpha := \alpha_{\trec}^{\ell}$ for successful iterations. 

\textbf{Proof:} See Appendix~\ref{pPro2}.
\end{proposition}
\begin{theorem}[\textbf{Local bound on $\mathbf{\|g(x^\ell)\|}$}] 
\label{thm:gradnorm-bound}
Under Assumptions~$\ASn$, $\ALn$, and $\ANn$, suppose that the conditions
of Proposition~\ref{prop:grad-inner-sdc} hold and $v(x^\ell,p^\ell)$ is defined as in \eqref{eq:wxp}. Then
\begin{equation} \label{eq:gnormbdup}
		\|g(x^\ell)\| \le  v(x^\ell, p^\ell)\,\Gamma(\alpha), \qquad \text{where} \qquad 		
        \Gamma(\alpha) :=
		\frac{\alpha}{\delta}\Big(\gamma_e\big(\wt\beta + \frac{L}{2}\delta^2\big)
		+ \frac{2\omega}{\alpha^2}\Big),
\end{equation}

\textbf{Proof:} See Appendix~\ref{pThe1}.
\end{theorem}
\begin{lemma}[\textbf{Noise-balanced scale crossing}]\label{lem:scale-crossing}
Under Assumptions~$\ASn$, $\ALn$, and $\ANn$, suppose that the conditions
of Proposition~\ref{prop:grad-inner-sdc} hold. Define the noise-balanced scale $a_\omega$ by
\begin{equation}\label{eq:a_omega}
a_\omega :=  \kappa_\star \sqrt{{\omega}/{L}}, \qquad \text{where }\; \kappa_\star := \sqrt{{2c_\alpha}/({\tilde{\beta}/L + \delta^2/2})},
\end{equation}
with $c_\alpha > 1$ and initial auxiliary factor $a^0$ satisfying \eqref{eq:initial-search-range} and
\begin{equation}\label{eq:aomega-reachability} 
a^0\geq a_\omega\geq\alpha_{\min}. 
\end{equation}
If \texttt{DAES} does not terminate, it eventually reaches the first unsuccessful iteration $\ell_\omega$ at scale $a^{\ell_\omega}$ such that
\begin{equation}\label{eq:scale_crossing_interval}
\rho_u a_\omega < a^{\ell_\omega} \le a_\omega, \qquad \rho_u\in (0,1).
\end{equation}

\textbf{Proof:} See Appendix~\ref{pLemma2}.
\end{lemma}
\begin{corollary}[\textbf{Worst-case stationarity bound}]
\label{coll}
Under Assumptions~$\ASn$, $\ALn$, and $\ANn$, let the conditions
of Theorem~\ref{thm:gradnorm-bound} hold for the unsuccessful iteration $\ell_\omega$ in Lemma~\ref{lem:scale-crossing}, and
\begin{equation}\label{eq:Gamma-kappa-bound}
C_d^0:= \delta^{-1}
\left[
\gamma_e\overline{\kappa}^\prime
\left(
{\tilde{\beta}/L + \delta^2/2}
\right)
+
{2}/{\underline{\kappa}^\prime}
\right],
\qquad
\underline{\kappa}^\prime
:=
{\rho_u\kappa_\star}/{c_\alpha},
\quad
\overline{\kappa}^\prime
:=
\kappa_\star.
\end{equation}
If, in addition, the angle condition \eqref{eq:angle_condition} holds, then 
\begin{equation}\label{eq:Gamma-kappa-boundNORM2} 
\|g(x^{\ell_\omega})\|
\leq
\Delta_{\angle}^{-1} C_d^0 \sqrt{L\omega}.  
\end{equation}

\textbf{Proof:} See Appendix~\ref{pChor1}.
\end{corollary}
\begin{theorem}[\textbf{Global bound on $\mathbf{\|g(x^\ell)\|}$}]
\label{thm:gradnorm-bound2}
Under Assumptions~$\ASn$, $\ALn$, and $\ANn$, let the worst-case iteration $\ell_{\omega}$ be defined as in Lemma~\ref{lem:scale-crossing}, and the conditions
of Proposition~\ref{prop:grad-inner-sdc} hold. If the angle condition
\eqref{eq:angle_condition} holds, then 
\begin{equation}\label{e.upnormgomegaangle}
\min_{\ell \in \mathcal K_\omega}\|g(x^\ell)\|
=
\Oc\big(
\Delta_{\angle}^{-1}\sqrt{L\omega}
\big), \qquad \mathcal K_\omega
:=
\{\ell\in\mathbb N_0:0\leq\ell\leq\ell_\omega< \ell_{\max}\},
\end{equation} 
where $\ell_{\max}$ denotes the prescribed maximum number of iterations.

\textbf{Proof:} See Appendix~\ref{pThe2}.
\end{theorem}
\begin{theorem}[\textbf{Error Bounds}] \label{t.c-noise1}
Under Assumptions~$\ASn$, $\ALn$, and $\ANn$, let the conditions
of Theorem~\ref{thm:gradnorm-bound2} hold, $\ul f:=f(x)$ at the minimizer $\ul x$ and $\ell^{**} \in \argminx_{\ell\in\mathcal{K}_\omega} \Vert{}g(x^\ell)\Vert{}$. If the angle condition
\eqref{eq:angle_condition} holds, then     
\begin{enumerate}[label=(\roman*)]
	\item if \( f \) is convex
	\begin{equation}\label{e.diffconvex1}
		f(x^{\ell^{**}}) - \ul f
		=
		\Oc\big(\Delta_{\angle}^{-1}\sqrt{L\omega}\big);
	\end{equation}
		
	\item if \( f \) is \( \mu_c \)-strongly convex, 
	\begin{equation}\label{e.diffsconvex1}
		f(x^{\ell^{**}}) - \ul f 
		=
		\Oc\big(\Delta_{\angle}^{-2}
		{L\omega}\big), 
		\qquad \text{and} \qquad
		\|x^{\ell^{**}} - \ul x\| 
		=
		\Oc\big(\Delta_{\angle}^{-1}
		\sqrt{L\omega}\big).
	\end{equation}
\end{enumerate}

\textbf{Proof:} See Appendix~\ref{pThe3}.
\end{theorem}
\begin{proposition}[\textbf{Efficiency}]\label{lem:noisy-eff-academic}
Under Assumptions~$\ASn$, $\ALn$, and $\ANn$, suppose that
the conditions of Proposition~\ref{prop:grad-inner-sdc} and the angle
condition \eqref{eq:angle_condition} hold at a successful iteration. Then, for the constant $\beta\in(0,1)$ defined in \eqref{eq:Amax-bound}, it holds:
\begin{equation}\label{neweffangle}
f(x)-f(x+s\alpha p)
\geq
\frac{\beta}{4K_1^2}\|g(x^\ell)\|^2,
\qquad
\text{where}
\qquad
K_1 :=
\frac{\gamma_e\left(\wt\beta+L\delta^2/2\right)}
{\Delta_{\angle}\delta}.
\end{equation}

\textbf{Proof:} See Appendix~\ref{pPro3}.
\end{proposition}

Defining the noise-limited accuracy as $\varepsilon_\omega$ and $C_d^0$ as in \eqref{eq:Gamma-kappa-bound}, we specify the \textbf{\textit{termination iteration}} $\ell_T$ as the first  iteration satisfying 
\begin{equation}\label{e.varomega}
\|g(x^\ell)\|\le\varepsilon_\omega, \qquad \varepsilon_\omega:= C_d\sqrt{L\omega} \ge \Delta_{\angle}^{-1} C_d^0 \sqrt{L\omega},
\end{equation}
where $C_d>0$ is fixed independently of $\omega$. This formally justifies the noise-scale unsuccessful iteration $\ell_\omega$ in Lemma~\ref{lem:scale-crossing} as the worst-case termination bound ($\ell_T \le \ell_\omega$).

\begin{theorem}[\textbf{Complexity Bounds}] \label{t.comf-noise1}   
Under Assumptions~$\ASn$, $\ALn$, $\ANn$, and the condition \eqref{eq:aomega-reachability}, let the hypotheses of
Proposition~\ref{prop:grad-inner-sdc} and the angle condition
\eqref{eq:angle_condition} hold before termination with the noise-limited accuracy $\varepsilon_\omega$ defined in \eqref{e.varomega}. Then, given $N_{sel}$ selection populations and maximum $N_e$ extrapolation steps, the number of function evaluations is bounded by 
\begin{equation}\label{eq:total-global-evaluation-count}
N_f
\leq
1+
\left(
\lambda N_{\mathrm{sel}}
+
2
+
N_e
\right)\ell_T, \qquad \text{and}
\end{equation}
\begin{enumerate}[label=(\roman*), itemsep=1ex]
    \item $\ell_T=\Oc\left(L^2\varepsilon_\omega^{-2}\right)$ and
    $N_f=\Oc\left(\varepsilon_\omega^{-2}\right)$ for
    {nonconvex} functions;

    \item $\ell_T=\Oc\left(\varepsilon_\omega^{-1}\right)$ and
    $N_f=\Oc\left(\varepsilon_\omega^{-1}\right)$ for
    {convex} functions;
		
    \item $\ell_T=\Oc\left(\log(\varepsilon_\omega^{-1})\right)$ and
    $N_f=\Oc\left(\log(\varepsilon_\omega^{-1})\right)$ for
    $\mu_c$-{strongly convex} functions. 
\end{enumerate}

\textbf{Proof:} See Appendix~\ref{pThe4}.
\end{theorem}
\subsection{Complexity with Randomized Directions}\label{sec:comResults2}

We extend the analysis in Subsection~\ref{sec:comResults} to the randomized setting, where the search direction is a triangular recombination direction. To establish high-probability convergence, we first construct the formal probabilistic framework. Next, we prove that Gaussian sampling yields a probabilistic angle property (Proposition~\ref{prop.angle.gauss.corrected}). Following that, we establish the directional alignment of $p_{\trec}^\ell$ via a three-stage hierarchy, leveraging two weighting mechanisms described in Subsection~\ref{seb.ptrec}. Finally, we ensure that {\tt DAES} drives the gradient norm down to the noise-accuracy threshold with high probability (Theorem~\ref{the.gradnormRan}), provide objective and iterate accuracy guarantees across all convexity regimes (Theorem~\ref{t.c-noise1r}) and obtain the overall high-probability evaluation complexity for {\tt DAES} (Theorem~\ref{t.comf-noise2}). 

\subsubsection{Probabilistic Sampling and Ranking Stability} 
\label{sec:random-selection}
\label{sec:randomized-assumptions}
At each iteration $\ell$, \texttt{DAES} generates $N_{\mathrm{sel}}$ 
independent Gaussian populations indexed by
$q\in[N_{\mathrm{sel}}]$. In the sorting--selection phase, it retains the one maximizing the observable within-group separation $\Delta_{\mathrm{sel}}^{\ell,q}$ among those populations. For theoretical purposes, we define the \textit{\textbf{analytical certificate-success event}} $\mathcal C_{\ell,q}$ and its complementary \textit{\textbf{analytical certificate-failure event}} $E_{\mathrm{sel},\ell,q}$ as
\begin{equation}\label{eq:selection-certificate}
\mathcal C_{\ell,q}
:=
\left\{
\Delta_{\mathrm{sel}}^{\ell,q}>2\omega
\right\},
\qquad 
\text{and}
\qquad
E_{\mathrm{sel},\ell,q}
:=
\mathcal C_{\ell,q}^{\,c}
=
\left\{
\Delta_{\mathrm{sel}}^{\ell,q}\leq2\omega
\right\}.
\end{equation}

While $\mathcal{C}_{\ell,q}$ is an event with respect to Gaussian sampling, it yields a deterministic ranking-stability certificate under bounded noise $\omega$ as
\begin{equation}\label{e.certi}
f_\pi^{r\ell}
-
f_\pi^{j_k^\star\ell}
\geq
\wt f_\pi^{r\ell}
-
\wt f_\pi^{j_k^\star\ell}
-
2\omega
>0,
\end{equation}
for every $r\in\mathcal G_\pi^k\setminus\{j_k^\star\}$, where $j_k^\star\in\mathcal G_\pi^k$ for $k\in[3]$. This guarantees that the noisy best-ranked element of each group is strictly superior in terms of the true objective $f$. 

By \eqref{eq:qsel}, $\Delta_{\mathrm{sel}}^{\ell,q^\star} \ge \Delta_{\mathrm{sel}}^{\ell,q}$ for all $q \in[N_{\mathrm{sel}}]$. If at least one of the $N_{\mathrm{sel}}$ populations satisfies $\mathcal C_{\ell,q}$, then the selected population $q^\star$ also satisfies $\mathcal C_{\ell,q^\star}$.  Accordingly, the \textit{\textbf{selected certificate-failure event}} $E_{\mathrm{sel},\ell}$ occurs  when all populations fail on $\mathcal{C}_{\ell,q}$, i.e., 
\begin{equation}\label{eq:selection-failure-event}
E_{\mathrm{sel},\ell}
:=
\left\{
\Delta_{\mathrm{sel}}^{\ell,q^\star}\leq2\omega
\right\}
= \left( \bigcup_{q=1}^{N_{\mathrm{sel}}} \mathcal{C}_{\ell,q} \right)^c
= \bigcap_{q=1}^{N_{\mathrm{sel}}}
E_{\mathrm{sel},\ell,q}.
\end{equation}
Serving as an analytical tool, $E_{\mathrm{sel},\ell}$ merely indicates that \eqref{e.certi} is unavailable, not an incorrect noisy ranking. It is neither evaluated by {\tt DAES} nor assigned a probability bound. Crucially, the probabilistic guarantees stem strictly from Gaussian sampling, while Assumption $\ANn$ serves solely as an analytical tool under a bounded deterministic noise level. Hence, conditional on sampling a well-aligned direction, all subsequent steps (sorting--selection and weighting) proceed deterministically. 

For any candidate population indexed by $q \in [N_{\mathrm{sel}}]$ at iteration $\ell$, let $d^{j\ell}$ with $\|d^{j\ell}\|=1$ denote the generated direction corresponding to the $j$-th column of $D_\pm^\ell$. For $g(x^\ell) \neq 0$, the \textbf{\textit{sampling-failure event}} for the $q$-th population is defined as
\begin{equation}\label{eq:sampling-failure-event}
E_{\mathrm{samp},\ell,q}
:=
\left\{
\max_{j\in[\lambda/2]}
\frac{
|g(x^\ell)^T d^{j\ell}|
}{
\|g(x^\ell)\|
}
< t_a
\right\},
\end{equation}
where $t_a > 0$ is the baseline alignment threshold. This event represents the realization in which no sampled direction in the $q$-th population achieves the required alignment with the true gradient.  If $g(x^\ell) = 0$, we set $E_{\mathrm{samp},\ell,q} := \emptyset$. 

Conversely, $E_{\mathrm{samp},\ell,q}^c$ defines the \textbf{\textit{sampling-success event}}, on which there exists at least one index $j \in [\lambda/2]$ in population $q$ satisfying
\begin{equation}
|g(x^\ell)^T d^{j\ell}| \ge t_a \|g(x^\ell)\|.
\end{equation}

To guarantee that sampling success holds uniformly across the entire iteration horizon (covering all $\ell_{\max} \times N_{\mathrm{sel}}$ sampled populations), the \textbf{\textit{global good event}} $E_G$ as the joint occurrence of sampling success is defined as
\begin{equation}\label{e.global-good-event}
E_G := \bigcap_{\ell=0}^{\ell_{\max}-1} \bigcap_{q=1}^{N_{\mathrm{sel}}} E_{\mathrm{samp},\ell,q}^c.
\end{equation}
To rigorously analyze the probability of $E_G$ under sequential population generation, let $\mathcal{F}_{\ell,q-1}$ denote the $\sigma$-algebra generated by the algorithmic history prior to generating the $q$-th population at iteration $\ell$, with $\mathcal{F}_{\ell,0} := \mathcal{F}_{\ell-1}$, where the initial $\sigma$-algebra is generated by the initial point $x^0$. It ensures that $x^\ell$ and $g(x^\ell)$ are $\mathcal{F}_{\ell,q-1}$-measurable, while the fresh Gaussian samples are independent of $\mathcal{F}_{\ell,q-1}$. It provides the filtration necessary to apply the tower property of conditional expectation when evaluating the joint failure probability ($\mathbb P(E_G^c)$) in the subsequent proofs.

\subsubsection{Probabilistic Angle Conditions}
\label{sec:randomDirections}
In our randomized analysis, we first bound the conditional failure probability of the Gaussian sampling mechanism for any individual population. As a consequence, we establish that on the global good event $E_G$, every candidate population generated across the entire iteration horizon contains at least one direction sufficiently aligned with the negative true gradient with high probability.
\begin{proposition}\label{prop.angle.gauss.corrected}
Let the symmetric normalized directions $\{d^{j\ell}\}_{j \in [\lambda]}$ be generated by \texttt{DAES} for $0 \leq \ell < \ell_{\max}$, where the population size $\lambda$ is calibrated as in \eqref{eq:lambda-confidence}. Then, for each of the $N_{\mathrm{sel}}$ generated populations with confidence budget $\gamma_{\mathrm{pop}}$, the conditional sampling-failure probability is at most $\gamma_{\mathrm{pop}}$ given the $\sigma$-algebra $\mathcal F_{\ell,q-1}$; i.e., 
\begin{equation}\label{eq:per-iteration-sampling-failure-prop}
\mathbb P\!\left(
E_{\mathrm{samp},\ell,q}
\,\middle|\,
\mathcal F_{\ell,q-1}
\right)
\leq
\gamma_{\mathrm{pop}}, \qquad \gamma_{\mathrm{pop}} := \frac{\gamma_0}
{\ell_{\max}N_{\mathrm{sel}}},
\end{equation}
where $\gamma_0 \in (0, 1/2)$ is given as the overall confidence parameter and used for selecting $\lambda$. Consequently, on the global good event $E_G$, there exists a sampled direction for which the angle alignment property 
\begin{equation}\label{e.angleqp.corrected}
\frac{|g(x^\ell)^T d^{j\ell}|}{\|g(x^\ell)\|} \geq t_a, \quad j\in[\lambda]
\end{equation}
holds with probability at least $1-\gamma_0$, where the baseline alignment threshold $t_a := \tau / \sqrt{n}$ is defined as in \eqref{eq:barepsilon-def} with $n\geq 2$ and $0<\tau<\sqrt{\pi/2}$.
\end{proposition}

\begin{proof}
 Let $\mathcal F_{\ell,q-1}$ be the sigma-algebra generated by the algorithmic history prior to generating the $q$-th population at iteration $\ell$. Let us fix $0\leq\ell<\ell_{\max}$ and $q\in[N_{\mathrm{sel}}]$ and denote the $n$-dimensional unit sphere by $\mathbb{S}^{n-1}$. By the
construction, $x^\ell$ and $g(x^\ell)$ are
$\mathcal F_{\ell,q-1}$-measurable, whereas the fresh Gaussian vectors
$\{z^{j\ell}\}_{j\in[\lambda/2]}$ are independent of
$\mathcal F_{\ell,q-1}$. They are conditionally i.i.d.\ and uniformly distributed on
$\mathbb S^{n-1}$. Let $u^\ell
:={g(x^\ell)}/{\|g(x^\ell)\|}$. The following finite-dimensional spherical-cap estimate follows the
arguments in \cite{Gratton2015}, letting $T
:=
(u^\ell)^T\widehat z^{j\ell}$.

First, suppose that $n\geq3$. Then $T$ has density
\[
f_n(t)
=
c_n(1-t^2)^{(n-3)/2},
\qquad
c_n
:=
\frac{\Gamma(n/2)}
{\sqrt{\pi}\Gamma((n-1)/2)},
\qquad
-1\leq t\leq1.
\]
Since $(n-3)/2\geq0$, the density $f_n$ is nonincreasing on $[0,1]$.
Therefore, with $t_a=\tau/\sqrt n$,
\[
\begin{aligned}
\mathbb P\!\left(
0\leq T<t_a
\,\middle|\,
\mathcal F_{\ell,q-1}
\right)
&=
\int_0^{\tau/\sqrt n}
c_n(1-t^2)^{(n-3)/2}\,dt
\leq
c_n\frac{\tau}{\sqrt n}.
\end{aligned}
\]
By the log-convexity of the Gamma function,
\[
\Gamma\!\left(x+\frac12\right)^2
\leq
\Gamma(x)\Gamma(x+1)
=
x\Gamma(x)^2,
\qquad
x>0.
\]
Taking $x=(n-1)/2$ gives
\[
\frac{\Gamma(n/2)}
{\Gamma((n-1)/2)}
\leq
\sqrt{\frac{n-1}{2}},
\ \ \text{and hence} \ \ 
\frac{c_n}{\sqrt n}
\leq
\sqrt{
\frac{n-1}{2n\pi}
}
<
\frac{1}{\sqrt{2\pi}}.
\]
Using the symmetry of $T$, we obtain
\[
\begin{aligned}
\mathbb P\!\left(
T\geq t_a
\,\middle|\,
\mathcal F_{\ell,q-1}
\right)
&=
\frac12
-
\mathbb P\!\left(
0\leq T<t_a
\,\middle|\,
\mathcal F_{\ell,q-1}
\right)
\geq
\frac12-\frac{\tau}{\sqrt{2\pi}}.
\end{aligned}
\]

Now, suppose that $n=2$. In this case, $T$ has density
\[
f_2(t)
=
\frac{1}{\pi\sqrt{1-t^2}},
\qquad
-1<t<1.
\]
Since $0<\tau<\sqrt{\pi/2}$, we have
$0<\tau/\sqrt2<1$, and therefore
\[
\mathbb P\!\left(
T\geq\frac{\tau}{\sqrt2}
\,\middle|\,
\mathcal F_{\ell,q-1}
\right)
=
\frac12
-
\frac{1}{\pi}
\arcsin\!\left(
\frac{\tau}{\sqrt2}
\right).
\]
Since $\arcsin(t)$ is convex on $[0,1]$ and joins the points
$(0,0)$ and $(1,\pi/2)$,
\[
\arcsin(t)
\leq
\frac{\pi}{2}t,
\qquad
0\leq t\leq1, \ \ \text{so that} \ \ 
\frac{1}{\pi}
\arcsin\!\left(
\frac{\tau}{\sqrt2}
\right)
\leq
\frac{\tau}{2\sqrt2}
\leq
\frac{\tau}{\sqrt{2\pi}},
\]
where the last inequality follows from $\pi \le 4$. Since $t_a = \tau/\sqrt{n} \le \tau/\sqrt{2}$ for all $n \ge 2$, 
\begin{equation*}
\mathbb P\!\left(
T \ge t_a
\,\middle|\,
\mathcal F_{\ell,q-1}
\right)
\ge
\frac12 - \frac{\tau}{\sqrt{2\pi}}.
\end{equation*}
Using symmetry, complementation and the conditional independence of $\{\widehat z^{j\ell}\}_{j\in[\lambda/2]}$, we have
\begin{equation*}
\mathbb{P}\left(E_{\mathrm{samp},\ell,q} \,\middle|\, \mathcal{F}_{\ell,q-1}\right) \le r_\tau^{\lambda/2},
\qquad
r_\tau := \tau\sqrt{2/\pi} \in (0,1).
\end{equation*}
From \eqref{eq:lambda-confidence}, we have ${\lambda}\log(1/r_\tau)/2 \ge \log(1/\gamma_{\mathrm{pop}})$, thereby proving \eqref{eq:per-iteration-sampling-failure-prop}. On $E_{\mathrm{samp},\ell,q}^c$, by the antipodal construction of $D_\pm^\ell$ in \eqref{eq.Dstore}, there exists at least one index $j \in [\lambda]$ for which \eqref{e.angleqp.corrected} holds with conditional probability at least $1 - \gamma_{\mathrm{pop}}$ given $\mathcal{F}_{\ell,q-1}$. To prove that the probability guarantee holds simultaneously for all nonstationary iterations, consider the global good event defined in \eqref{e.global-good-event}. By the union bound and tower property of conditional expectation,
\begin{align*}
\mathbb{P}(E_G^c)
\le
\sum_{\ell=0}^{\ell_{\max}-1} \sum_{q=1}^{N_{\mathrm{sel}}} \mathbb{E}\left[\mathbb{P}\left(E_{\mathrm{samp},\ell,q} \,\middle|\, \mathcal{F}_{\ell,q-1}\right)\right]
\le
\ell_{\max} N_{\mathrm{sel}} \gamma_{\mathrm{pop}}
=
\gamma_0,
\end{align*}
which establishes that $\mathbb{P}(E_G) \ge 1 - \gamma_0$ and completes the proof.
\end{proof}


On $E_{\mathrm{samp},\ell,q}^{c}$ for any fixed candidate population, the antipodal construction of $D_{\pm}^{\ell}$ ensures that some direction satisfies $g(x^\ell)^Td^{j\ell}\leq -t_a\|g(x^\ell)\|$. For the retained population selected as in \eqref{eq:qsel}, if it achieves sampling success, its \textbf{\textit{latent descent direction}} $d^{j^\star\ell}$ satisfies
\begin{equation}\label{e,jstar}
g(x^\ell)^T d^{j^\star\ell} \le -t_a \|g(x^\ell)\|, 
\qquad
\text{where}
\quad
j^\star \in \arg\min_{j\in[\lambda]} g(x^\ell)^T d^{j\ell}.
\end{equation}
This is an \textbf{\textit{analytical descent-angle inequality}}, distinct from the algorithmic test \texttt{NDF-SDC} in \eqref{e.SDC-nograd}.  

We establish the directional alignment of $p_{\trec}^\ell$ via a three-stage hierarchy, leveraging two weighting mechanisms designed according to the algorithmic conditions \eqref{eq:barepsilon-def} and \eqref{e.marginCon}: first, bounded-noise ordering transfers a sufficient descent angle from the latent (unobservable) direction $d^{j^\star\ell}$ to $d_\pi^{1\ell}$ (Proposition~\ref{pro:ranking-separation}); second, the within-group dominance parameter $\varepsilon_{\bar{w}}$ \eqref{eq:barepsilon-def} preserves alignment inside $\mathcal G_\pi^1$, transferring it to $d_{\rec 1}^\ell$ (Lemma~\ref{lem:winner-dominant-angle}); third, the cross-group margin property \eqref{e.marginCon} prevents cancellation from secondary groups, securing the overall alignment of $p_{\trec}^\ell$ (Proposition~\ref{prop:angle-transfer-trec}).


\begin{proposition}\label{pro:ranking-separation} 
Suppose that Assumptions~$\ALn$ and $\ANn$ hold, and define
\begin{equation}\label{eq:selection-transfer-constant}
G_{\mathrm{sel}}
:=
\frac{2\alpha_{\max}}{\alpha_{\min}t_a} B_{\mathrm{sel}},
\qquad
B_{\mathrm{sel}}
:=
\frac{L}{2} \left( \alpha_{\max} + \frac{\alpha_{\max}^2}{\alpha_{\min}} \right) + \frac{2\omega}{\alpha_{\min}},
\end{equation}
where $t_a$ is given as in \eqref{e.angleqp.corrected}, and consider the retained population selected as in \eqref{eq:qsel}. If this population achieves sampling success, then for any iteration $0 \le \ell < \ell_{\max}$ satisfying the nonstationarity condition 
\begin{equation}\label{eq.gradGsel}
    \|g(x^\ell)\| > G_{\mathrm{sel}},
\end{equation}
the best noisy-ranked direction $d_\pi^{1\ell}$ achieves the alignment guarantee
\begin{equation}\label{eq.descGlobal}
    g(x^\ell)^Td_\pi^{1\ell}
<
-t_{\mathrm{sel}}\|g(x^\ell)\|,
\end{equation}
with the particular alignment threshold $t_{\mathrm{sel}}>0$, which directly yields
\begin{equation}\label{e.p4a}
\frac{|g(x^\ell)^T d_\pi^{1\ell}|}{\|g(x^\ell)\|} \ge t_{\mathrm{sel}}, \qquad \text{where} \quad t_{\mathrm{sel}} := \frac{\alpha_{\min}}{2\alpha_{\max}} t_a.
\end{equation}
Consequently, on the global good event $E_G$, which occurs with probability at least $1 - \gamma_0$, the angle alignment property \eqref{e.p4a} holds simultaneously for all iterations satisfying \eqref{eq.gradGsel}.
\end{proposition}
\begin{proof}
Let $d^{j^\star\ell}$ be the (unobservable) direction in the 
retained population that achieves the theoretical alignment guarantee \eqref{e,jstar}, and let $i^\star$ denote its ranked position. However, {\tt DAES} identifies the search direction $d_\pi^{1\ell}$ via noisy function evaluations $\wt f$. Let $\wt f_\pi^{1\ell}$ be the observed value at the globally best noisy-ranked direction $d_\pi^{1\ell}$. Thus,
$\wt f\!(
x^\ell+\alpha_1d_\pi^{1\ell}
)
\leq
\wt f\!(
x^\ell+\alpha_\star d_\pi^{i^\star\ell}
)$, where 
$\alpha_1$ and
$\alpha_\star$ are the mutation step-sizes along $d_\pi^{1\ell}$ and $d_\pi^{i^\star\ell}$, respectively. 
By $\ANn$, $f\!(
x^\ell+\alpha_1d_\pi^{1\ell}
)
\leq
f\!(
x^\ell+\alpha_\star d_\pi^{i^\star\ell}
)
+2\omega$.
 Using \eqref{e.TayErr1} with $\hat s= d_\pi^{1\ell}$ and $\hat s =d_\pi^{i^\star\ell}$, respectively, 
\begin{align*}
    f\!\left(
x^\ell+\alpha_1d_\pi^{1\ell}
\right)
&\geq
f(x^\ell)
+
\alpha_1g(x^\ell)^Td_\pi^{1\ell}
-
\frac{L}{2}\alpha_1^2,\qquad \text{and}\\
f\!\left(
x^\ell+\alpha_\star d_\pi^{i^\star\ell}
\right)
&\leq
f(x^\ell)
+
\alpha_\star g(x^\ell)^Td_\pi^{i^\star\ell}
+
\frac{L}{2}\alpha_\star^2.
\end{align*}
Combining these last three inequalities, and using \eqref{e,jstar}, $\|d_\pi^{i^\star\ell}\|=1$ and $\alpha_1,\alpha_\star
\in
[\alpha_{\min},\alpha_{\max}]$, yields $g(x^\ell)^Td_\pi^{1\ell}
\leq
-2t_{\mathrm{sel}}\|g(x^\ell)\|
+
B_{\mathrm{sel}}$ followed by \eqref{eq.descGlobal} because $B_{\mathrm{sel}}<t_{\mathrm{sel}}\|g(x^\ell)\|$ if $\|g(x^\ell)\|>G_{\mathrm{sel}}$. Hence \eqref{e.p4a} follows immediately. On $E_G$, the sampling-success property holds for every candidate population and therefore for the selected population $q^\star$. The deterministic argument above then establishes \eqref{e.p4a} for every iteration satisfying \eqref{eq.gradGsel}. On $E_G$, consequently, by the union bound and tower property, \eqref{e.p4a} holds with overall probability at least $1 - \gamma_0$.
\end{proof}

As formalized in Lemma~\ref{lem:winner-dominant-angle} below, each group-representative recombination direction retains a positive alignment margin with the normalized gradient whenever its best direction is well-aligned. Since, by construction of \texttt{DAES}, the best directions of ranked groups are associated with $j_1^\star = 1$, $j_2^\star = \mu+1$, and $j_3^\star = 2\mu + 1$, this property applies in particular to the representative recombination direction of the first group, i.e., $d_{\mathrm{rec}\,1}$. Because its associated best direction $d_\pi^{1 \ell}$ is well-aligned from Proposition~\ref{pro:ranking-separation}, it provides the condition required for the final result in Proposition~\ref{prop:angle-transfer-trec}.
\begin{lemma}\label{lem:winner-dominant-angle}
For each group $\mathcal{G}_\pi^k$, let $\{d_\pi^{j\ell}\}_{j \in \mathcal{G}_\pi^k}$ be the set of sorted unit directions from the retained population selected as in \eqref{eq:qsel}, and let $\{\bar{w}_j\}_{j \in \mathcal{G}_\pi^k}$ be their corresponding within-group recombination weights generated according to \eqref{eq:wd-algorithmic-weights}, given a within-group dominance parameter $\varepsilon_{\bar{w}}$ satisfying \eqref{eq:barepsilon-def}.
If the best direction of each group achieves the alignment guarantee 
$$|g(x^\ell)^T d_\pi^{j_k^\star \ell}| \ge t_{\mathrm{sel}} \|g(x^\ell)\|, \qquad k \in [3],$$
with the particular alignment threshold $t_{\mathrm{sel}}>0$, then each group-representative recombination direction $d_{\mathrm{rec}\,k}^\ell$ with the within-group alignment margin $\bar{\varepsilon} > 0$ satisfies
\begin{equation}\label{eq:wd-strict-positivity}
\frac{|g(x^\ell)^T d_{\mathrm{rec}\,k}^\ell|}{\|g(x^\ell)\|} \ge \bar{\varepsilon}, \quad k \in [3], \qquad \text{where} \quad
\bar{\varepsilon} := (1 - \varepsilon_{\bar{w}})t_{\mathrm{sel}} - \varepsilon_{\bar{w}}.
\end{equation}
\end{lemma}
\begin{proof}
Let $q^\ell:= {g(x^\ell)}/{\|g(x^\ell)\|}$ and apply the reverse triangle inequality to $(q^\ell)^T d^\ell_{\mathrm{rec}\,k}$. Using \eqref{e.RecDirs} and the fact that $
\|q^\ell\|\,\|d^{j\ell}_\pi\|
=
1$, we have $|(q^\ell)^Tv^\ell_k|
\ge
(1-\varepsilon_{\bar w})
|(q^\ell)^T  d^{j_k^\star\ell}_\pi|
-
\varepsilon_{\bar w}$ where $\quad \sum_{j} \bar{w}_j \le \varepsilon_{\bar{w}}$ with $j \in \mathcal{G}_\pi^k \setminus \{j_k^\star\}$. In particular, if
$|(q^\ell)^T d^{j_k^\star\ell}_\pi|\ge t_{\mathrm{sel}}$, applying the definition of $\bar\varepsilon$ in \eqref{eq:wd-strict-positivity} completes the proof. 
\end{proof}
\begin{proposition}\label{prop:angle-transfer-trec}%
Under the Assumptions~$\ALn$ and $\ANn$, consider the retained population selected as in \eqref{eq:qsel}. Let the normalized within-group weights and  triangular weights be generated according to \eqref{eq:wd-algorithmic-weights} and \eqref{eq:theta-margin-weights}, respectively, with the cross-group margin property \eqref{e.marginCon}. Then, on the global good event  $E_G$, which occurs with probability at least $1 - \gamma_0$ given the overall confidence budget $\gamma_0$ used to define the population-size $\lambda$ in \eqref{eq:lambda-confidence}, the triangular recombination direction $p_{\mathrm{trec}}^\ell$ achieves the alignment guarantee
\begin{equation}\label{e.p4b}
\frac{|g(x^\ell)^T p_{\trec}^\ell|}{\|g(x^\ell)\|\, \|p_{\trec}^\ell\|} \ge \Delta_r, \qquad \text{where} \quad \Delta_r := \frac{(1-\rho)\bar\varepsilon}{1+\rho\bar\varepsilon} > 0,
\end{equation}
with $\rho \in (0,1)$ and $\bar\varepsilon$ defined in \eqref{eq:wd-strict-positivity}, simultaneously for all iterations satisfying \eqref{eq.gradGsel}. Consequently, the probability statement holds for the reciprocal alignment ratio
\begin{equation}\label{e.anglenew}
v(x^\ell,p_{\mathrm{trec}}^\ell)
\leq \Delta_r^{-1}.
\end{equation}
\end{proposition}
\begin{proof}
Given $\gamma_0$ and the retained population, the global good event $E_G$ with $\mathbb{P}(E_G) \ge 1 - \gamma_0$ guarantees sampling success across all nonstationary iterations; let $q^\ell:= {g(x^\ell)}/{\|g(x^\ell)\|}$. On $E_G$, applying Proposition~\ref{pro:ranking-separation}, $|(q^\ell)^T d_\pi^{1\ell}| \ge t_{\mathrm{sel}}$ and applying Lemma~\ref{lem:winner-dominant-angle}, $|(q^\ell)^T d_{\mathrm{rec}\,1}^\ell| \ge \bar{\varepsilon}$. Now, recall the unscaled triangular recombination direction \eqref{e.trid}. Applying the reverse triangle inequality to $(q^\ell)^T \widehat p_{\mathrm{trec}}^\ell$, then using $|(q^\ell)^T d_{\mathrm{rec}\,k}^\ell| \le \|q^\ell\| \|d_{\mathrm{rec}\,k}^\ell\| \le 1$ and $\|d_{\mathrm{rec}\,k}^\ell\| \le 1$ for all $k \in [3]$, and taking the algorithmic cross-group margin property \eqref{e.marginCon} yields
\begin{equation}\label{eq:proof-step2-pre}
|(q^\ell)^T \widehat p_{\mathrm{trec}}^\ell| 
\ge \frac{\big( \bar{\varepsilon} \theta_1^\ell \alpha_{\mathrm{rec}\,1}^\ell - \sum_{k\neq 1} \theta_k^\ell \alpha_{\mathrm{rec}\,k}^\ell \big)}{\Theta^\ell} > 
 \frac{(1-\rho)\bar{\varepsilon} \theta_1^\ell \alpha_{\mathrm{rec}\,1}^\ell}{\Theta^\ell} > 0, \qquad \Theta^\ell := \sum_{k=1}^3 \theta_k^\ell.
\end{equation}
This strict positivity implies $\widehat p_{\mathrm{trec}}^\ell \neq 0$. Combining $\|\widehat p_{\mathrm{trec}}^\ell\| \le \frac{1}{\Theta^\ell} \sum_{k=1}^3 \theta_k^\ell \alpha_{\mathrm{rec}\,k}^\ell$ with \eqref{eq:proof-step2-pre},
\begin{equation*}
\frac{|g(x^\ell)^T \widehat p_{\mathrm{trec}}^\ell|}{\|g(x^\ell)\| \|\widehat p_{\mathrm{trec}}^\ell\|} \ge 
\frac{(1-\rho)\bar{\varepsilon} \theta_1^\ell \alpha_{\mathrm{rec}\,1}^\ell}{\sum_{k=1}^3 \theta_k^\ell \alpha_{\mathrm{rec}\,k}^\ell}\ge \frac{(1-\rho)\bar{\varepsilon}}{1 + r^\ell}, \qquad \text{where} \quad r^\ell := \frac{\sum_{k\neq 1} \theta_k^\ell \alpha_{\mathrm{rec}\,k}^\ell}{\theta_1^\ell \alpha_{\mathrm{rec}\,1}^\ell}.
\end{equation*}
Since $r^\ell < \rho \bar{\varepsilon}$ by \eqref{e.marginCon}, the alignment ratio is bounded below by $\Delta_r$. Since directional alignment is invariant under unit-normalization in $p_{\mathrm{trec}}^\ell$, \eqref{e.p4b} follows immediately. To complete the proof, note that the strict positivity \eqref{eq:proof-step2-pre} ensures $|(q^\ell)^T p_{\mathrm{trec}}^\ell|>0$. Using \eqref{eq:wxp} and taking the reciprocal of the angular ratio in \eqref{e.p4b} gives \eqref{e.anglenew}, and the probability statement follows.
\end{proof}
\begin{lemma}\label{lem.Gsel}
    Under Assumptions~$\ALn$ and $\ANn$, suppose that the condition in Proposition~\ref{pro:ranking-separation} holds and let the noise-scale feasibility condition \eqref{eq:aomega-reachability} hold. For all iterations satisfying \eqref{eq.gradGsel}, the stationarity transfer threshold is tightly controlled by 
    \begin{equation}\label{eq:Gsel-noise-scale}
       G_{\mathrm{sel}} \le C_{\mathrm{sel}}\sqrt{L\omega}, 
       \qquad \text{where} \quad 
       C_{\mathrm{sel}} := {2R_\alpha}t_a^{-1} C^0_{\mathrm{sel}}, 
    \end{equation}
    with 
    \begin{equation}
        C^0_{\mathrm{sel}} = \left[ \frac{\kappa_\star}{2}\left(R_\alpha + (R_\alpha)^2\right) + \sqrt{\frac{2\wt\beta}{L}} \right], \qquad
        R_\alpha:=\frac{\alpha_{\max}}{\alpha_{\min}}.
    \end{equation}
\end{lemma}
\begin{proof}
From \eqref{eq:aomega-reachability} in  Lemma~\ref{lem:scale-crossing}, $\alpha_{\min} \le a_\omega = \kappa_\star \sqrt{\omega/L}$. Moreover, from $\ANn$, $\alpha_{\min}^2 > {2\omega/\wt\beta}$. Thus, $B_{\mathrm{sel}} \le  
C^0_{\mathrm{sel}}
\sqrt{L\omega}$. Dividing by $t_{\mathrm{sel}}$ \eqref{e.p4a} completes the proof. 
\end{proof}
\subsubsection{Probabilistic Bounds and Complexity}
\label{sec:ProbBoundConv}

In this subsection, we first establish the high-probability global gradient bound (Theorem~\ref{the.gradnormRan}) and subsequently translate it into noise-limited accuracy (Theorem~\ref{t.c-noise1r}). Next, we present a high-probability counterpart of Theorem~\ref{t.comf-noise1} (Theorem~\ref{t.comf-noise2}).

To obtain the stationarity level in the randomized setting under \eqref{eq:aomega-reachability}, we utilize Proposition~\ref{prop:angle-transfer-trec} and Lemma~\ref{lem.Gsel}. If $\|g(x^\ell)\| \le G_{\mathrm{sel}}$ at any iteration $\ell$, the stationarity bound is directly satisfied. Otherwise, the randomized angle condition holds on the event $E_G$ with parameter $\Delta_r$ via Proposition~\ref{prop:angle-transfer-trec}. Extending the deterministic worst-case bound \eqref{eq:Gamma-kappa-boundNORM2} to the randomized setting yields $\|g(x^\ell)\| \le \Delta_r^{-1} C_d^0 \sqrt{L\omega}$. Thus, the overall worst-case stationarity level is governed by the maximum of both bounds with an $\omega$-independent constant $\widehat C_r$ (and thus any $C_r$), leading to
\begin{equation}\label{eq:epsilon-omega-r}
\varepsilon_{\omega,r} := C_r\sqrt{L\omega}, \qquad \text{where} \quad C_r \ge \widehat C_r := \max\left\{ C_{\mathrm{sel}},\, \Delta_r^{-1} C_d^0 \right\}.
\end{equation}

\begin{theorem}\label{the.gradnormRan}
Under Assumptions~$\ALn$, $\ANn$, and $\ASn$, let the conditions of Proposition~\ref{prop:angle-transfer-trec} be satisfied, and the hypotheses of Theorem~\ref{thm:gradnorm-bound2} hold for the search directions $p^\ell_{\trec}$ under \eqref{eq:aomega-reachability} before termination. For the randomized stationarity level $\varepsilon_{\omega,r}$ in \eqref{eq:epsilon-omega-r}, it holds on $E_G$ that
\begin{equation}\label{e.upnormgomega1}
\min_{\ell\in\mathcal K_\omega}\|g(x^\ell)\| = \Oc\!\left(\varepsilon_{\omega,r}\right), \qquad  \text{with probability at least $1-\gamma_0$ }.
\end{equation}
\end{theorem}
\begin{proof}
By definition of $\varepsilon_{\omega,r}$, whether $\|g(x^{\ell_\omega})\| \le G_{\mathrm{sel}}$ or $\|g(x^{\ell_\omega})\| > G_{\mathrm{sel}}$, it holds that $\|g(x^{\ell_\omega})\| \le \varepsilon_{\omega,r}$ on $E_G$. Since $\ell_\omega \in \mathcal{K}_\omega$, $\min_{\ell\in\mathcal{K}_\omega}\|g(x^\ell)\| \le \|g(x^{\ell_\omega})\| \le \varepsilon_{\omega,r}$, which yields \eqref{e.upnormgomega1}. The probabilistic statement then follows from $\mathbb{P}(E_G) \ge 1-\gamma_0$, completing the proof.
\end{proof}
\begin{theorem}\label{t.c-noise1r}
Suppose that Assumptions~$\ALn$, $\ANn$, and $\ASn$ hold, the conditions of Theorem~\ref{the.gradnormRan} are satisfied, and the hypotheses of Theorem~\ref{t.c-noise1} hold for the triangular recombination search directions. Then, the following conclusions hold on $E_G$ with probability at least $1-\gamma_0$:
\begin{enumerate}[label=(\roman*)]
\item If $f$ is convex, then
\begin{equation}\label{e.pdiffconvex}
f(x^{\ell^{**}})-\ul f
=
\Oc\!\left(
\varepsilon_{\omega,r}
\right).
\end{equation}

\item If $f$ is $\mu_c$-strongly convex, then
\begin{equation}\label{e.pdiffsconvex}
f(x^{\ell^{**}})-\ul f
=
\Oc\!\left(
{\varepsilon_{\omega,r}^{\,2}}
\right),
\qquad
\|x^{\ell^{**}}-\ul x\|
=
\Oc\!\left(
{\varepsilon_{\omega,r}}
\right).
\end{equation}
\end{enumerate}
\end{theorem}
\begin{proof}
Theorem~\ref{the.gradnormRan} guarantees $\min_{\ell\in\mathcal K_\omega}\|g(x^\ell)\| \le \varepsilon_{\omega,r}$, which implies $\|g(x^{\ell^{**}})\| \le \varepsilon_{\omega,r}$ by the definition of $\ell^{**}$. The bounds in terms of $\varepsilon_{\omega,r}$ then follow directly from the properties of $f$ as established in Theorem~\ref{t.c-noise1} (see Appendix~\ref{pThe3}). Since $\mathbb P(E_G) \ge 1-\gamma_0$, the high-probability conclusions follow immediately, completing the proof.
\end{proof}


\begin{theorem}\label{t.comf-noise2}
Under Assumptions~$\ALn$, $\ANn$, $\ASn$, and the noise-scale feasibility condition \eqref{eq:aomega-reachability}, let the conditions of Theorem~\ref{the.gradnormRan} be satisfied, and the hypotheses of Theorem~\ref{t.comf-noise1} hold before termination. Then, the following conclusions hold on $E_G$ with probability at least $1-\gamma_0$:
\begin{enumerate}[label=(\roman*), itemsep=1ex]
    \item $\ell_T=\Oc\left(L^2\varepsilon_{\omega,r}^{-2}\right)$ and
    $N_f=\Oc\left(L^2\varepsilon_{\omega,r}^{-2}\right)$ for
    {nonconvex} functions;

    \item $\ell_T=\Oc\left(L^2\varepsilon_{\omega,r}^{-1}\right)$ and
    $N_f=\Oc\left(L^2\varepsilon_{\omega,r}^{-1}\right)$ for
    {convex} functions;
		
    \item $\ell_T=\Oc\left(L^2\log(\varepsilon_{\omega,r}^{-1})\right)$ and
    $N_f=\Oc\left(L^2\log(\varepsilon_{\omega,r}^{-1})\right)$ for
    $\mu_c$-{strongly convex} functions. 
\end{enumerate}
\end{theorem}

\begin{proof}
The proof follows the same line of reasoning as the deterministic complexity result (Theorem~\ref{t.comf-noise1}; see Appendix~\ref{pThe4}), replacing the deterministic parameters with their randomized counterparts.

Specifically, on $E_G$, for any pre-termination iteration $0 \le \ell < \ell_T \le \ell_\omega < \ell_{\max}$, it holds that $\|g(x^\ell)\| > \varepsilon_{\omega,r}$, where $\varepsilon_{\omega,r}$ is defined in \eqref{eq:epsilon-omega-r}. By definition of $\varepsilon_{\omega,r}$, two crucial properties are simultaneously guaranteed: first, $\|g(x^\ell)\| > \varepsilon_{\omega,r} \ge C_{\mathrm{sel}}\sqrt{L\omega}$, ensuring the randomized angle condition holds  with parameter $\Delta_r$ via Proposition~\ref{prop:angle-transfer-trec}, and second, $\|g(x^\ell)\| > \varepsilon_{\omega,r} \ge \Delta_r^{-1} C_d^0 \sqrt{L\omega}$. Thus, for any successful iteration $\ell \in S$, Proposition~\ref{lem:noisy-eff-academic} applies with $\Delta_r$, guaranteeing the sufficient decrease $f(x^\ell) - f(x^{\ell+1}) \ge \tilde{\eta} \|g(x^\ell)\|^2  > \tilde{\eta}  \varepsilon_{\omega,r}^2$, where $\tilde{\eta} \propto L^{-2}\Delta_r^2$. Summing these decreases over $S$ yields the bound on the number $|S|$ of successful iterations for nonconvex functions. For convex and strongly convex functions, the proof follows in an analogous manner to the deterministic analysis, replacing $\varepsilon_{\omega} $ with $\varepsilon_{\omega,r}$. Similarly, for unsuccessful iterations, the analysis follows directly from Theorem~\ref{t.comf-noise1}, where $a^\ell > a_\omega$ holds for all $\ell < \ell_T$, which bounds the unsuccessful iterations by $|U|$. Combining $\ell_T=|S|+|U|$ on $E_G$ with $\mathbb P(E_G) \ge 1-\gamma_0$ completes the proof.
\end{proof}

Table~\ref{tab:randomized-complexity-comparison} summarizes the principal nonconvex function-evaluation complexity results for representative randomized
DFO methods. Based on this table:
\begin{itemize}
\item While noiseless methods (e.g., \texttt{DSPFD}, \texttt{STP}, and \texttt{VRBBO}) target an arbitrary tolerance $\varepsilon > 0$, noise-aware algorithms like \texttt{VRDFON} (deterministic) and \texttt{DAES} (randomized) adopt accuracy limits $\varepsilon_{\omega,d} = \sqrt{L\omega}$ and $\varepsilon_{\omega,r}:= C_r\sqrt{L\omega}$ defined in \eqref{eq:epsilon-omega-r}, respectively. Under condition~\eqref{eq:aomega-reachability}, the randomized threshold $\varepsilon_{\omega,r}$ in our  \texttt{DAES} matches the deterministic noise-order bound in \texttt{VRDFON}, up to the randomized scaling constant $C_r$.
\item Sampling strategies vary across the baseline methods (see the respective references in Table~\ref{tab:randomized-complexity-comparison}). Specifically, \texttt{DAES} evaluates $N_{\mathrm{sel}}$ independent populations of $\lambda$ directions (formed as $\lambda/2$ antipodal pairs), using an evaluation budget of $\lambda N_{\mathrm{sel}}$ per outer iteration, where $\lambda$ is calibrated via \eqref{eq:lambda-confidence}.

\begin{remark}
In \texttt{DAES}, the dimension factor $\Delta_r^{-2} = \Theta(n)$ stems directly from the alignment threshold $t_a = \tau/\sqrt{n}$, which governs $t_{\mathrm{sel}}$, $\bar{\varepsilon}$, and consequently $\Delta_r$. Since the stationarity level satisfies $\varepsilon_{\omega,r} \le \max\{C_{\mathrm{sel}}, C_r\}\sqrt{L\omega}$, with both $C_{\mathrm{sel}} \propto t_a^{-1} = \mathcal{O}(\sqrt{n})$ and $C_r \propto \Delta_r^{-1} = \mathcal{O}(\sqrt{n})$, this yields the explicit factor $n$ in the evaluation complexity $\mathcal{O}\big(n L^2 (\sqrt{L\omega})^{-2}\big)$, which is equivalent to the bound reported in Table~\ref{tab:randomized-complexity-comparison}. By comparison, \texttt{VRDFON}~\cite{VRDFON} defines its stationarity parameter directly as $\varepsilon_{\omega,d} := \sqrt{nL\omega}$, leading to a complexity of $\mathcal{O}\big(n L^2 (\sqrt{nL\omega})^{-2}\big)$, matching the entry displayed in Table~\ref{tab:randomized-complexity-comparison}.
\end{remark}

\end{itemize}
\begin{table}[t]
\centering
\caption{Nonconvex function-evaluation complexity of randomized DFO methods.}
\label{tab:randomized-complexity-comparison}
\scriptsize
\setlength{\tabcolsep}{4pt}
\renewcommand{\arraystretch}{1.2}

\begin{tabularx}{\textwidth}{
l
l 
>{\raggedright\arraybackslash}X
>{\raggedright\arraybackslash}X
}
\toprule
Method
&
Guarantee in
&
Direction budget
&
Complexity
\\
\midrule

\texttt{DSPFD}~\cite{Gratton2015}
&
high-probability 
&
$m=2$ admissible
&
$\mathcal{O}(nL^2\varepsilon^{-2})$
\\

{\tt STP}~\cite{bergou2020stochastic}
&
expectation
&
One sampled direction
&
$\mathcal{O}(nL\varepsilon^{-2})$
\\
{\tt VRBBO}~\cite{VRBBO}
&
high-probability 
&
$R=\lceil\log_2(\eta_0^{-1})\rceil$
&
$\mathcal{O}(nL^2\varepsilon^{-2})$\\
\hline
{\tt VRDFON}~\cite{VRDFON}
&
high-probability 
&
$R=\lceil T_0^{-1}\log_2(\eta_0^{-1})\rceil$
&
$\mathcal{O}(L^2\varepsilon_{\omega,d}^{-2})$
\\
{\tt DAES}
&
high-probability 
&
$\lambda N_{\mathrm{sel}}$
&
$\mathcal{O}(L^2\varepsilon_{\omega,r}^{-2})$
\\
\bottomrule
\end{tabularx}
\end{table}

\section{Numerical Experiments}\label{NumSec}

This section presents an ablation study of \texttt{DAES} and compares it with  \texttt{VRDFON}~\cite{VRDFON}, \texttt{VRBBO}~\cite{VRBBO}, \texttt{NMSMAX}~\cite{NMSMAX}, \texttt{MADFO}~\cite{MADFO}, and \texttt{DFNDFL}~\cite{DFNDFL}.  These experiments do not aim to claim overall superiority over advanced solvers like \texttt{MADFO}, which integrate empirical heuristics. Instead, they measure how effectively a design built for provable complexity bounds performs in practice. The experimental setup is structured as follows.

\textbf{\textit{Test Problems.}} We consider 230 test problems with the dimensions $2 \leq n \leq 30$ from the \texttt{BCP} collection of the {\tt BARON} software package~\cite{BARON}, available at \url{https://minlp.com/optimization-test-problems}. Each deterministic test problem is perturbed only by absolute uniform noise (see \cite{MADFO}) and evaluated at eight noise levels from \(\{10^i\}_{i=-5}^{2}\).

\textbf{\textit{Comparing Metrics.}} We use \textit{data profiles}~\cite{MorW} and \textit{performance profiles}~\cite{DolM}. Data profiles show how many problems a solver can solve as the available computational budget increases, while performance profiles show how often a solver performs close to the best solver on each problem. Let $\mathcal S$ be the set of solvers, and $\mathcal P$ and $\Omega$ denote the 230 test problems and the eight noise levels, respectively, and $\mathcal I=\mathcal P\times\Omega$ with $|\mathcal I| = 1840$.
For $i\in \mathcal I$, let $c_{i,s}$ be the function-evaluation cost
for solver $s$, with unsuccessful runs assigned infinite cost, and let $n_i$
denote the dimension of problem $p$.  The data profile ($\delta_s(\kappa)$) and the performance profile ($\rho_s(\bar\tau)$) are defined as
\begin{equation*}
 \delta_s(\kappa):=\frac{1}{|\mathcal I|}
 \left|\left\{i\in\mathcal I:\frac{c_{i,s}}{n_i}\leq\kappa\right\}\right|, \qquad
 \rho_s(\bar\tau):=\frac{1}{|\mathcal I|}
 \left|\left\{i\in\mathcal I:
 \frac{c_{i,s}}{\min_{r\in\mathcal S}c_{i,r}}\leq\bar\tau
 \right\}\right|.
\end{equation*}

\textbf{\textit{Stopping Criteria. }}
For each solver \(s\in\mathcal S\), the ratio
$q_s=(f_s-f_{\tar})/(f_0-f_{\tar})$ measures the \textit{performance}, where \(f_s\) is the best value returned by the solver, \(f_0\) is a starting value, and \(f_{\tar}\) is a target evaluated at a high-quality local minimizer, or potentially an approximate global minimizer, identified by sequentially applying a set of DFO solvers. The ratio is used only for benchmarking since \(f_{\tar}\) is generally unavailable in real-world problems. Solver \(s\) solves a problem if \(q_s\le\varepsilon_q\) before reaching \texttt{nfmax} or \texttt{secmax}; otherwise, it is deemed unsuccessful. We set \(\texttt{secmax}=600\) sec, \(\texttt{nfmax}=20000\), and \(\varepsilon_q=10^{-4}\).

\textbf{\textit{Initial Point.}} Following~\cite{VRDFON,MADFO}, the nominal starting point is shifted by $\bar \xi \in \mathbb{R}^n$ with components $\bar \xi_i := (-1)^{i-1}\frac{2}{i+2}$ for $i \in [n]$. This shift avoids trivial initialization advantages on \texttt{BCP} benchmark problems, ensuring a more realistic performance evaluation. Setting $y^0 := 0$, all noisy function evaluations are computed at $y^\ell + \bar \xi$ for $\ell \ge 0$.

\textbf{\textit{Tuning Parameters. }} For \texttt{DAES}, we use the parameter values summarized in
Table~\ref{tab:daes-tuning}. We set \(N_{\mathrm{sel}}=1\) and \(\ell_{\max}=2\times10^4\). Accordingly, $\lambda
  =18$, where the population size $\lambda$ is obtained by \eqref{eq:lambda-confidence}. 
In our recombination phase, the normalized within-group weights of each group \(\mathcal{G}_k\), \(k\in[3]\) are defined as in \eqref{eq:wd-algorithmic-weights}, whose within-group dominance parameter $\varepsilon_{\bar{w}}$ satisfying \eqref{eq:barepsilon-def} is set as $\varepsilon_{\bar{w}}
=
\xi{t_{\mathrm{sel}}}/{(1+t_{\mathrm{sel}})}$ with $\xi>0$.  Normalizing the group weights in \eqref{e.RecDirs} ensures that each recombination direction reflects its internal ranking independently of the weight magnitudes across groups. All other solvers are run with their default parameter settings; see \cite{MADFO}.

\begin{table}[t]
\centering
\caption{Tuning parameters used in \texttt{DAES}.}
\label{tab:daes-tuning}
\footnotesize
\begin{tabular}{lll}
\toprule
Parameter & Value & Role \\
\midrule
\(\eta\) & $0.99$ & triangular recombination parameter \\
\(\delta\) & $2$ & norm of the triangular search direction \\
\(\gamma_e\) & $4$ & extrapolation expansion factor \\
\(\sigma^0\) & $1$ & initial step-size scale \\
\(\widetilde{\beta}\) & $10^{-12}$ & {\tt NDF-SDC} forcing parameter \\
\(\alpha_{\min}\) & $10^{-5}$ & minimum step-size scale \\
\(\alpha_{\max}\) & $10^{2}$ & maximum step-size scale \\
\(\rho_u\) & $0.99$ & unsuccessful-iteration contraction factor \\
\(c_\alpha\) & $2$ & initial signed-search safeguard \\
\(\tau\) & $0.25$ & within-group dominance parameter \\
\(\xi\) & $0.5$ & within-group weight scaling parameter \\
\(\rho\) & $0.5$ & parameter in the cross-group margin property \\
\(a^0\) & $1$ & initial auxiliary signed-search scale \\
\bottomrule
\end{tabular}
\end{table}
\subsection{Noise-Sensitivity and Ablation Study}\label{subsec:noise-protocol}

To evaluate solver performance under noise, we report two complementary metrics: \emph{robustness}, defined as the total number of problem instances solved prior to reaching the evaluation limit, and \emph{efficiency}, measured by the normalized $n_f$-efficiency statistic representing the computational evaluation cost per solved instance. The evaluated algorithm variants are structured as follows. 

The variants \texttt{DAES} and \texttt{DAES2}--\texttt{DAES4} retain the complete algorithm while systematically varying the population parameter $N_{\mathrm{sel}} \in \{1, 2, 3, 4\}$, respectively, with \texttt{DAES} ($N_{\mathrm{sel}}=1$) serving as the baseline implementation. \texttt{DAES5}--\texttt{DAES7} isolate the contributions of key algorithmic mechanisms: \texttt{DAES5} removes symmetric sampling to assess the role of antipodal directions; \texttt{DAES6} replaces the proposed triangular recombination scheme with a traditional weighted recombination mechanism over the top $\mu = \lfloor\lambda/2\rfloor$ directions (see, e.g., \cite{MADFO}); \texttt{DAES7} deactivates the extrapolation component to assess its effect on convergence acceleration.

In Figure~\ref{fig:nsel-noise}, the left
panel measures robustness, and the right panel gives the reported normalized
$n_f$-efficiency.  At the smallest noise level $10^{-5}$, the baseline
\texttt{DAES} solves 189 problems, compared with 183, 180, and 167 for
\texttt{DAES2}--\texttt{DAES4}, respectively. At the largest level $10^2$, the counts are 86, 91, 92, and 85.
Thus, higher $N_{\mathrm{sel}}$ values do not yield uniform or monotonic robustness improvements across noise levels. The right panel reveals a clear trend: the baseline achieves the highest $n_f$-efficiency across all tested noise levels. For instance, at a noise level of $10^{-5}$, the baseline reaches an efficiency of approximately $0.261$, compared to $0.205$, $0.202$, and $0.178$ for \texttt{DAES2}--\texttt{DAES4}, respectively. A similar advantage persists under severe noise.

\begin{figure}[t]
\centering
\resizebox{.98\linewidth}{!}{%
\begin{tikzpicture}
\begin{groupplot}[
 group style={group size=2 by 1,horizontal sep=1.55cm,group name=nselnoise},
 daesnoiseaxis,
 ]
\nextgroupplot[ymin=80, ymax=200,
ylabel={Number of solved problems},
 legend to name=nselnoise-legend,legend columns=4,
 legend style={draw=none,fill=none,font=\scriptsize,column sep=8pt,cells={anchor=west}}]
\addplot+[curveA] coordinates {(1e-05,189) (0.0001,186) (0.001,179) (0.01,175) (0.1,138) (1,128) (10,105) (100,86)}; \addlegendentry{\texttt{DAES}}
\addplot+[curveB] coordinates {(1e-05,183) (0.0001,185) (0.001,178) (0.01,172) (0.1,150) (1,126) (10,108) (100,91)}; \addlegendentry{\texttt{DAES2}}
\addplot+[curveC] coordinates {(1e-05,180) (0.0001,179) (0.001,176) (0.01,161) (0.1,150) (1,130) (10,108) (100,92)}; \addlegendentry{\texttt{DAES3}}
\addplot+[curveD] coordinates {(1e-05,167) (0.0001,171) (0.001,164) (0.01,166) (0.1,150) (1,123) (10,105) (100,85)}; \addlegendentry{\texttt{DAES4}}

\nextgroupplot[ymin=0.15, ymax=0.30,
ylabel={$n_f$ efficiency}]
\addplot+[curveA] coordinates {(1e-05,0.2610403214023952) (0.0001,0.2688638243575971) (0.001,0.2764926487704412) (0.01,0.2750806273020177) (0.1,0.2764724432324876) (1,0.2684478059103221) (10,0.25387114585726) (100,0.2189011979147178)};
\addplot+[curveB] coordinates {(1e-05,0.2050603970309426) (0.0001,0.2168040084271318) (0.001,0.2354524613254322) (0.01,0.2481499358960655) (0.1,0.2542968375465761) (1,0.2477383016903196) (10,0.2192583359375242) (100,0.2163664479860727)};
\addplot+[curveC] coordinates {(1e-05,0.2020429159604948) (0.0001,0.2007980031235262) (0.001,0.2076348905859081) (0.01,0.2283312295462284) (0.1,0.2225753037570931) (1,0.2287779473486828) (10,0.2094645909662672) (100,0.2070821974378657)};
\addplot+[curveD] coordinates {(1e-05,0.1779529382775584) (0.0001,0.1889592792257024) (0.001,0.194324814125759) (0.01,0.2147552704128095) (0.1,0.2122024921251012) (1,0.1869059690036933) (10,0.2048126706580883) (100,0.1792298966874639)};
\end{groupplot}
\node[anchor=north] at ([yshift=-1.0cm]nselnoise c1r1.south east)
 {\pgfplotslegendfromname{nselnoise-legend}};
\end{tikzpicture}%
}
\caption{Sensitivity of the complete \texttt{DAES} method to $N_{\rm sel}$ as
the noise magnitude increases.}
\label{fig:nsel-noise}
\end{figure}

In Figure~\ref{fig:version-noise}, the left
panel shows that \texttt{DAES} and \texttt{DAES5}--\texttt{DAES7} retain comparable solved counts,
although their rankings change with noise.  Hence, higher noise levels decrease the number of solved problems across all variants. The right panel highlights the efficiency trade-offs among the structural variants. Notably, the complete \texttt{DAES} framework maintains an almost stable $n_f$-efficiency profile across (almost) all tested noise levels, demonstrating strong efficiency robustness against increasing noise. While removing triangular recombination (\texttt{DAES6}) or symmetric sampling (\texttt{DAES5}) slightly increases efficiency under specific low-to-moderate noise levels, removing extrapolation (\texttt{DAES7}) systematically reduces efficiency across all noise regimes. Overall, the baseline design provides a balanced, noise-insensitive trade-off between robustness and computational efficiency, rather than strict pointwise dominance in every scenario.

\begin{figure}[t]
\centering
\resizebox{.98\linewidth}{!}{%
\begin{tikzpicture}
\begin{groupplot}[
 group style={group size=2 by 1,horizontal sep=1.55cm,group name=vernoise},
 daesnoiseaxis,
 ]
\nextgroupplot[ ymin=80, ymax=200,
ylabel={Number of solved problems},
 legend to name=vernoise-legend,legend columns=4,
 legend style={draw=none,fill=none,font=\scriptsize,column sep=9pt,cells={anchor=west}}]
\addplot+[curveA] coordinates {(1e-05,189) (0.0001,186) (0.001,179) (0.01,175) (0.1,138) (1,128) (10,105) (100,86)}; \addlegendentry{\texttt{DAES}}
\addplot+[curveB] coordinates {(1e-05,188) (0.0001,186) (0.001,185) (0.01,166) (0.1,148) (1,133) (10,107) (100,83)}; \addlegendentry{\texttt{DAES5}}
\addplot+[curveC] coordinates {(1e-05,187) (0.0001,184) (0.001,178) (0.01,169) (0.1,143) (1,121) (10,97) (100,84)}; \addlegendentry{\texttt{DAES6}}
\addplot+[curveD] coordinates {(1e-05,187) (0.0001,184) (0.001,179) (0.01,167) (0.1,151) (1,124) (10,94) (100,87)}; \addlegendentry{\texttt{DAES7}}

\nextgroupplot[ymin=0.15, ymax=0.35,
ylabel={$n_f$ efficiency}]
\addplot+[curveA] coordinates {(1e-05,0.2610403214023952) (0.0001,0.2688638243575971) (0.001,0.2764926487704412) (0.01,0.2750806273020177) (0.1,0.2764724432324876) (1,0.2684478059103221) (10,0.25387114585726) (100,0.2189011979147178)};
\addplot+[curveB] coordinates {(1e-05,0.2688300294475437) (0.0001,0.2631665346935868) (0.001,0.2839478677977106) (0.01,0.2737951208842366) (0.1,0.3065337325233452) (1,0.292802624945499) (10,0.2774368492401988) (100,0.208797291274706)};
\addplot+[curveC] coordinates {(1e-05,0.2832449094803717) (0.0001,0.2816067521180761) (0.001,0.3157863875269832) (0.01,0.2953460938844849) (0.1,0.2859178008685499) (1,0.2687045230560059) (10,0.2524131891051775) (100,0.2380705101863891)};
\addplot+[curveD] coordinates {(1e-05,0.2227638431833336) (0.0001,0.2031817031224964) (0.001,0.214172583926319) (0.01,0.1952848298521321) (0.1,0.2385075604388859) (1,0.2138744957510691) (10,0.1942291590184485) (100,0.1804861783993207)};
\end{groupplot}
\node[anchor=north] at ([yshift=-1.0cm]vernoise c1r1.south east)
 {\pgfplotslegendfromname{vernoise-legend}};
\end{tikzpicture}%
}
\caption{Structural ablation of \texttt{DAES} across increasing noise.}
\label{fig:version-noise}
\end{figure}

\subsection{Comparison with Other DFO Solvers}\label{subsec:other-solvers-noise}

Figure~\ref{fig:others-noise} compares \texttt{DAES} with five
external derivative-free solvers.  On the left, \texttt{DAES} solves 189 and 86 problems at noise levels $10^{-5}$ and $10^2$, respectively, whereas \texttt{MADFO} solves 204 and 98. \texttt{MADFO} achieves the highest solved count across most noise levels. Nevertheless, \texttt{DAES} remains competitive with the other external
methods. Based on the right panel, under high noise regimes,
\texttt{DAES} is comparable with all the solvers. While \texttt{MADFO} exhibits greater overall robustness—consistent with its emphasis on practical robustness and more elaborate implementation—these curves alone do not identify which individual enhancements cause the performance gap.

\begin{figure}[t]
\centering
\resizebox{.98\linewidth}{!}{%
\begin{tikzpicture}
\begin{groupplot}[
 group style={group size=2 by 1,horizontal sep=1.55cm,group name=othersnoise},
 daesnoiseaxis,
]
\nextgroupplot[ymin=60, ymax=220,
ylabel={Number of solved problems},
 legend to name=othersnoise-legend,
 legend columns=3,
 legend style={draw=none,fill=none,font=\scriptsize,column sep=8pt,cells={anchor=west}}]
\addplot+[curveA] coordinates {(1e-05,189) (0.0001,186) (0.001,179) (0.01,175) (0.1,138) (1,128) (10,105) (100,86)}; \addlegendentry{\texttt{{DAES}}}
\addplot+[curveB] coordinates {(1e-05,172) (0.0001,164) (0.001,140) (0.01,122) (0.1,108) (1,91) (10,72) (100,64)}; \addlegendentry{\texttt{{DFNDFL}}}
\addplot+[curveC] coordinates {(1e-05,190) (0.0001,182) (0.001,160) (0.01,137) (0.1,115) (1,97) (10,76) (100,67)}; \addlegendentry{\texttt{{VRBBO}}}
\addplot+[curveD] coordinates {(1e-05,194) (0.0001,184) (0.001,178) (0.01,156) (0.1,135) (1,117) (10,89) (100,79)}; \addlegendentry{\texttt{{VRDFON}}}
\addplot+[curveE] coordinates {(1e-05,186) (0.0001,169) (0.001,149) (0.01,118) (0.1,96) (1,78) (10,70) (100,70)}; \addlegendentry{\texttt{{NMSMAX}}}
\addplot+[curveF] coordinates {(1e-05,204) (0.0001,199) (0.001,196) (0.01,172) (0.1,152) (1,127) (10,117) (100,98)}; \addlegendentry{\texttt{{MADFO}}}

\nextgroupplot[ymin=0.15, ymax=0.55,
ylabel={$n_f$ efficiency}]
\addplot+[curveA] coordinates {(1e-05,0.2610403214023952) (0.0001,0.2688638243575971) (0.001,0.2764926487704412) (0.01,0.2750806273020177) (0.1,0.2764724432324876) (1,0.2684478059103221) (10,0.25387114585726) (100,0.2189011979147178)};
\addplot+[curveB] coordinates {(1e-05,0.4579183559595993) (0.0001,0.4832631895311424) (0.001,0.3900964058453039) (0.01,0.3601680760577657) (0.1,0.3324550226492463) (1,0.2805681652364633) (10,0.2185295973595907) (100,0.2035244551483409)};
\addplot+[curveC] coordinates {(1e-05,0.4350645037940441) (0.0001,0.4397755193291017) (0.001,0.3992820377606897) (0.01,0.3652411964153371) (0.1,0.3275378992671206) (1,0.2895405908737865) (10,0.244325340191592) (100,0.2122809682019369)};
\addplot+[curveD] coordinates {(1e-05,0.5214514685664766) (0.0001,0.4798931862358486) (0.001,0.4733313084745966) (0.01,0.4156915009294429) (0.1,0.3845070256745905) (1,0.3449849310775834) (10,0.2598289162952433) (100,0.2357478185066538)};
\addplot+[curveE] coordinates {(1e-05,0.5180155019641012) (0.0001,0.4874571439734932) (0.001,0.4179161669729932) (0.01,0.3459179962168856) (0.1,0.2637215312882228) (1,0.2111302421573654) (10,0.1931616250011884) (100,0.189996853167598)};
\addplot+[curveF] coordinates {(1e-05,0.2880553646143977) (0.0001,0.2819379088314895) (0.001,0.3247515213233259) (0.01,0.2879441752124377) (0.1,0.3085232975938795) (1,0.2240063247468769) (10,0.2423137121773501) (100,0.2193899162755538)};
\end{groupplot}
\node[anchor=north] at ([yshift=-1.0cm]othersnoise c1r1.south east)
 {\pgfplotslegendfromname{othersnoise-legend}};
\end{tikzpicture}%
}
\caption{A comparison between \texttt{DAES} and the external solvers as a function of noise magnitude.}
\label{fig:others-noise}
\end{figure}

With $\mathcal S=\{\texttt{DAES},\texttt{MADFO}\}$,
Figure~\ref{fig:profiles-daes-madfo} aggregates all 1840 instances and compares the distributions of evaluation costs. The left panel shows a meaningful, but limited, evaluation-budget advantage for \texttt{DAES}. At $\kappa=5$, its data-profile fraction is approximately $0.258$, compared with $0.171$ for \texttt{MADFO}. It overtakes \texttt{DAES} at approximately $\kappa=33.4$. Hence, in terms of robustness, \texttt{DAES} is advantageous in an intermediate, small-budget regime.
The right panel shows that \texttt{MADFO} has greater relative efficiency, with a fastest-solver frequency approximately 9.2 percentage points higher than \texttt{DAES}, while its overall robustness is approximately 4.3 percentage points higher. These differences may reflect the heuristic techniques and transformation matrices used in \texttt{MADFO}, although their individual effects are not isolated by these experiments.

\begin{figure}[t]
\centering
\resizebox{.98\linewidth}{!}{%
%
}
\caption{Data profile (left) and performance profile (right) for 
\texttt{DAES} and \texttt{MADFO}.}
\label{fig:profiles-daes-madfo}
\end{figure}

\section{Conclusion}\label{ConclusionSec}

For noisy DFO, we proposed \texttt{DAES}, an \texttt{MAES}-type method that fixes the transformation matrix to the identity. By integrating symmetric sampling, grouped sorting--selection of all samples, and a novel triangular recombination mechanism, \texttt{DAES} constructs search directions that systematically mitigate two distinct sources of instability: noise-induced ranking distortion and directional cancellation. The resulting angle guarantee yields high-probability iteration and evaluation-complexity bounds for nonconvex,
convex, and strongly convex problems under bounded noise, matching the asymptotic orders of various noisy or noiseless methods referenced in Table~\ref{tab:randomized-complexity-comparison}. To the best of our knowledge, this represents the first 
 complexity guarantee of its kind within this algorithmic class.

Our empirical analysis highlights three main findings for \texttt{DAES}. Increasing $N_{\mathrm{sel}}$ yields no systematic efficiency gain, while a single population provides a simple baseline with analytical tractability under noise. Although $N_{\mathrm{sel}} > 1$ under \eqref{eq.jkq}--\eqref{eq:qsel} yields no performance gain, exploring alternative selection criteria could unlock its potential in future work. Ablating core components—even without guaranteed theoretical convergence—reveals a trade-off between noise sensitivity and efficiency, whereas the complete method exhibits relatively stable efficiency across the tested noise levels. Finally, compared with \texttt{MADFO}, \texttt{DAES} remains competitive in robustness under intermediate, small-budget regimes. Overall, with its simple structure, \texttt{DAES} bridges complexity guarantees and competitive empirical performance relative to \texttt{MADFO}.
\bigskip

\textbf{Acknowledgments. }{The authors kindly acknowledge Arnold Neumaier for his constructive comments and suggestions. Morteza Kimiaei acknowledges financial support from the Austrian Science Foundation under \url{https://doi.org/10.55776/PAT2747625}. Mahsa Yousefi is a member of the INdAM Research Group GNCS; her work was partially supported by INDAM-GNCS through Progetti di Ricerca 2026, whose support is gratefully acknowledged.}

\appendix
\section{Proofs in Deterministic Analysis}
\label{app:deterministic-proofs}
For simplicity of notation throughout most of the proofs, let $p:=p^\ell$ and $x:=x^\ell$.
\subsection{Proof of Lemma~\ref{lem:noisy-to-noiseless}}\label{pLemma1}
Using \eqref{e.muW-nograd} and $\mu_s(\alpha)$ in \eqref{eq:noiseless-satisfied}, 
$$\wt\mu_s(\alpha)-\mu_s(\alpha)
=
-\alpha^{-2}\big(\,
\wt f(x+s\alpha p)-f(x+s\alpha p) 
-
\big(\wt f(x)-f(x)\big) \,\big).$$ 
Using \eqref{e.noisecon}, we end up with 
$\left|
\wt\mu_s(\alpha)-\mu_s(\alpha)
\right|
\leq{2\omega}/{\alpha^2}$. Now, since $\wt\mu_s(\alpha)>\wt\beta$, we obtain $\mu_s(\alpha)>\beta$ and $\beta<\wt\beta<1$ from \eqref{eq:Amax-bound}. Moreover, \eqref{eq:noise-upper-bound} yields $\beta>0$, which completes the proof.

\subsection{Proof of Proposition \ref{prop:termination}}\label{pPro1}

 Let $\alpha := \gamma_e^t \alpha^{\ell 0}$ (for $t \ge 0$) denote any successful step-size during the extrapolation phase at any successful outer iteration $\ell$, i.e., satisfying $\wt\mu_s(\alpha) > \wt\beta$ for some $s \in \{-1,1\}$. By Lemma~\ref{lem:noisy-to-noiseless}, any such step satisfies the noiseless sufficient decrease $\mu_s(\alpha) > \beta$ with $\beta$ defined in \eqref{eq:Amax-bound}, which implies
$f(x^\ell_{\trec} + s \alpha p^\ell_{\trec}) < f(x^\ell_{\trec}) - \beta \alpha^2$. By induction on the successful iterates, we have $f(x^\ell_{\trec}) \le f(x^0)$. Combined with the lower bound $\underline{f} > -\infty$, this yields $\underline{f} \le f(x^\ell_{\trec} + s \alpha p^\ell_{\trec}) < f(x^0) - \beta \alpha^2$, which implies 
\begin{equation*}
\alpha < \sqrt{\beta^{-1}(f(x^0) - \underline{f})} =: A_{\max}.
\end{equation*}
Thus, any successful step-size must strictly satisfy $\alpha < A_{\max}$. Combining the following two cases for $A_{\max}$ relative to $\alpha_{\min}$ confirms the definition of $N_e$, completing the proof:
\begin{enumerate}
    \item {Let $A_{\max} \le \alpha_{\min}$:} 
    Since the initial signed search enforces $\alpha^{\ell 0} \ge \alpha_{\min}$ by \eqref{eq:initial-search-range}, any successful trial step would require $\alpha_{\min} \le \alpha^{\ell 0} < A_{\max} \le \alpha_{\min}$, which is a contradiction. Consequently, no successful initial signed step can trigger extrapolation in this branch, yielding $N_e = 0$.

    \item {Let $A_{\max} > \alpha_{\min}$:} 
    Since $\alpha^{\ell 0} \ge \alpha_{\min}$, we have $\gamma_e^t \alpha^{\ell 0} \ge \gamma_e^t \alpha_{\min}$. An extrapolation trial step must be unsuccessful as soon as $\gamma_e^t \alpha_{\min} \ge A_{\max}$. Solving $\gamma_e^t \ge A_{\max}/\alpha_{\min}$ for the integer index $t$ shows that an unsuccessful trial step occurs at or before
    \begin{equation*}
    t = \left\lceil \log_{\gamma_e} \left( {A_{\max}}/{\alpha_{\min}} \right) \right\rceil.
    \end{equation*}
    Including this final unsuccessful trial step, the total number of extrapolation expansion trials per successful outer iteration is bounded above by $N_e$ in \eqref{eq:Nexpand}.
\end{enumerate}



\subsection{Proof of Proposition \ref{prop:grad-inner-sdc}}\label{pPro2}
\noindent{\sc Case 1.}
Let $\alpha=\alpha^{\ell0}$. If both
$x+s\alpha p$, $s\in\{-1,1\}$, fail {\tt NDF-SDC}, then since
$\wt\mu_s(\alpha)\leq\wt\beta$,
\[
-\wt\beta\alpha^2
\leq
\wt f(x+s\alpha p)-\wt f(x)
\overset{\mathrm{\eqref{e.noisecon}}}{\leq}
f(x+s\alpha p)-f(x)+2\omega
\overset{\mathrm{\eqref{e.TayErr1}}}{\leq}
s\alpha g(x)^Tp
+
\frac{L}{2}\alpha^2\delta^2
+
2\omega.
\]
Hence,
\[
-s\,g(x)^Tp
\leq
\alpha
(\wt\beta+\frac{L}{2}\delta^2)
+
\frac{2\omega}{\alpha}, \qquad \text{and} \qquad
|g(x)^Tp|
\leq
\alpha
(\wt\beta+\frac{L}{2}\delta^2)
+
\frac{2\omega}{\alpha}.
\]
Here, applying the first inequality for both $s=+1$ and $s=-1$ gives the second one, followed immediately by \eqref{eq:upBound} for $\gamma_e>1$. 

\noindent{\sc Case 2.}
Let $\alpha:=\alpha_{\trec}^{\ell}$ be the last accepted successful
step-size along $sp$, for some $s\in\{-1,1\}$. Hence, by construction and Proposition~\ref{prop:termination}, the
accepted step is followed by the first unsuccessful expanded trial, so $\wt\mu_s(\alpha)>\wt\beta$ and 
$\wt\mu_s(\gamma_e\alpha)\leq\wt\beta$.

\noindent{\sc Case 2A.}
From $\wt\mu_1(\alpha)>\wt\beta$, $\wt f(x+\alpha p)-\wt f(x)<-\wt\beta\alpha^2$. Using \eqref{e.TayErr1} and \eqref{e.noisecon} implies
\[
\alpha g(x)^Tp-\frac{L}{2}\alpha^2\delta^2
\leq
f(x+\alpha p)-f(x)
\leq
\wt f(x+\alpha p)-\wt f(x)+2\omega
<
-\wt\beta\alpha^2+2\omega.
\]
Rearranging terms, we arrive at the following bound
\begin{equation}\label{e.cb1}
g(x)^Tp
<
-\wt\beta\alpha
+\frac{L}{2}\alpha\delta^2
+\frac{2\omega}{\alpha}
<
\gamma_e\alpha
\left(\wt\beta+\frac{L}{2}\delta^2\right)
+
\frac{2\omega}{\alpha}.
\end{equation}
Similarly, from the condition $\wt\mu_1(\gamma_e\alpha)\leq\wt\beta$, using \eqref{e.TayErr1} and \eqref{e.noisecon} yields an inequality. Dividing by $\gamma_e\alpha$ and rearranging its terms yields \eqref{e.cb2}, which when combined with \eqref{e.cb1} directly gives \eqref{eq:upBound}:
\begin{equation}\label{e.cb2}
-g(x)^Tp \leq \gamma_e\alpha \left(\wt\beta+\frac{L}{2}\delta^2\right) + \frac{2\omega}{\gamma_e\alpha} \leq \gamma_e\alpha \left(\wt\beta+\frac{L}{2}\delta^2\right) + \frac{2\omega}{\alpha}.
\end{equation}
\noindent{\sc Case 2B.}
Similarly, from $\wt\mu_{-1}(\alpha)>\wt\beta$ and $\wt\mu_{-1}(\gamma_e\alpha)\leq\wt\beta$, we end up with the following bounds, respectively, which, when combined, give \eqref{eq:upBound}: 
\begin{equation*}\label{e.cb3}
-g(x)^Tp
<
-\wt\beta\alpha
+\frac{L}{2}\alpha\delta^2
+\frac{2\omega}{\alpha}
<
\gamma_e\alpha
\left(\wt\beta+\frac{L}{2}\delta^2\right)
+
\frac{2\omega}{\alpha}.
\end{equation*}
\begin{equation*}\label{e.cb4}
g(x)^Tp
\leq
\gamma_e\alpha
\left(\wt\beta+\frac{L}{2}\delta^2\right)
+
\frac{2\omega}{\gamma_e\alpha}
\leq
\gamma_e\alpha
\left(\wt\beta+\frac{L}{2}\delta^2\right)
+
\frac{2\omega}{\alpha}.
\end{equation*}


\subsection{Proof of Theorem \ref{thm:gradnorm-bound}}\label{pThe1}

From \eqref{eq:wxp}, if $g(x)=0$, then $v(x,p)=1$, and
\eqref{eq:gnormbdup} follows immediately. Otherwise, if
$g(x)^Tp=0$, then $v(x,p)=+\infty$, and
\eqref{eq:gnormbdup} holds trivially. Now, let
$|g(x)^Tp|>0$. Using the upper bound in \eqref{eq:upBound}, we obtain the upper bound in \eqref{eq:gnormbdup} as:
	\begin{eqnarray*}
		\|g(x)\| 
		= v(x, p)  \frac{|g(x)^T p|}{\|p\|} 
		\le v(x, p) \Big( \frac{\gamma_e\alpha}{\|p\|} \big(\wt\beta+ \frac{L}{2}\|p\|^2 \big)
		+ \frac{2\omega}{\alpha\|p\|} \Big), \qquad \delta=\|p\|.
	\end{eqnarray*}


\subsection{Proof of Lemma~\ref{lem:scale-crossing}}\label{pLemma2}  
Let any fixed scale $a \in [\alpha_{\min}, \alpha_{\max}]$. At $a^\ell = a$, the initial step-size in \eqref{e.alpha0} satisfies $\alpha^{\ell 0} \ge a / c_\alpha$. By Lemma~\ref{lem:noisy-to-noiseless}, any successful iteration at scale $a$ yields a strictly positive decrease $f(x^\ell) - f(x^{\ell+1}) \ge \beta (a/c_\alpha)^2 > 0$. Since $f$ is bounded below by $\underline{f}$, infinitely many successful iterations at scale $a$ are impossible. Therefore, at any scale $a$, only finitely many successful iterations can occur before an unsuccessful trial takes place.

If $a^0 = a_\omega$, the first unsuccessful iteration trivially occurs at scale $a^{\ell_\omega} = a_\omega \in (\rho_u a_\omega, a_\omega]$. Now, consider $a^0 > a_\omega \ge \alpha_{\min}$. Since $a$ remains unchanged during successful iterations and strictly contracts via $a^{\ell+1} = \max\{\alpha_{\min}, \rho_u a^\ell\} < a^\ell$ at unsuccessful iterations whenever $a^\ell > a_\omega \ge \alpha_{\min}$, the sequence $\{a^\ell\}$ must eventually cross $a_\omega$. Let $a^-$ denote the last scale strictly greater than $a_\omega$ prior to the first contraction that brings the scale to or below $a_\omega$. Let $\ell_\omega$ be the specific unsuccessful iteration that triggers this contraction, yielding the updated scale $a^{\ell_\omega} := \max\{\alpha_{\min}, \rho_u a^-\}$. By definition of $a^-$, we have $a^{\ell_\omega} \le a_\omega$. Furthermore, using $a^- > a_\omega$ and $\rho_u > 0$, we obtain $a^{\ell_\omega} = \max\{\alpha_{\min}, \rho_u a^-\} \ge \rho_u a^- > \rho_u a_\omega$. Combining these bounds yields $\rho_u a_\omega < a^{\ell_\omega} \le a_\omega$. Thus, $\ell_\omega$ is the first unsuccessful iteration at which $a^{\ell_\omega} \in (\rho_u a_\omega, a_\omega]$, completing the proof.


\subsection{Proof of Corollary~\ref{coll}}\label{pChor1}
Let $\ell = \ell_\omega$. From
\eqref{eq:initial-search-range}, since $\rho_u \in (0,1)$, $c_\alpha^{-1}\rho_u a_\omega < 
\alpha^{\ell0}
\leq
a^\ell$. Using $\underline{\kappa}^\prime, \overline{\kappa}^\prime$ in 
\eqref{eq:Gamma-kappa-bound} and $a_\omega=\kappa_\star\sqrt{\omega/L}$ in \eqref{eq:a_omega}, $\underline{\kappa}^\prime
\sqrt{{\omega}/{L}}
<
\alpha^{\ell0}
\leq
\overline{\kappa}^\prime
\sqrt{{\omega}/{L}}$.
Applying this to $\Gamma(\alpha)$ in \eqref{eq:gnormbdup}, where $\alpha=\alpha^{\ell0}$, gives the bound $\|g(x^{\ell_\omega})\| \le  v_m\,\Gamma(\alpha)$, where $ \Gamma(\alpha)
\leq
C_d^0\sqrt{L\omega}$ and $v_m:= v(x^{\ell_\omega},p^{\ell_\omega})\leq\Delta_{\angle}^{-1}$.
If, in addition, \eqref{eq:angle_condition} holds, then \eqref{eq:Gamma-kappa-boundNORM2} follows immediately.
\subsection{Proof of Theorem \ref{thm:gradnorm-bound2}}\label{pThe2}
Over $\mathcal K_\omega$, $\min_{\ell \in \mathcal K_\omega}\|g(x^\ell)\| \leq \|g(x^{\ell_\omega})\|
\overset{\mathrm{\eqref{eq:Gamma-kappa-boundNORM2}}}{\leq}
\Delta_{\angle}^{-1} C_d^0 \sqrt{L\omega}$. Since $\underline{\kappa}^\prime$ and $\overline{\kappa}^\prime$ are constants independent of the noise level $\omega$ for fixed
$L$, $\delta$, $\wt\beta$ and $\gamma_e$ in $C_d^0$, we obtain \eqref{e.upnormgomegaangle}; see Remark~\ref{rem:asymptotic-convention}. This concludes the proof.

\subsection{Proof of Theorem \ref{t.c-noise1}}\label{pThe3}
By the true sufficient decrease guaranteed on successful iterations and
the fact that unsuccessful iterations leave the iterate unchanged,
$\{x^\ell\}_{\ell\in\mathcal K_\omega}\subseteq\mathcal L(x^0)$.
Moreover, the global minimizer $\ul x\in\mathcal L(x^0)$.
Since $\mathcal L(x^0)$ is compact, there exists a constant
$\ol r>0$ such that $\|x^\ell-\ul x\|
\leq
\ol r$ for $
\ell\in\mathcal K_\omega$. 

Since $\|g(x^{\ell^{**}})\|\leq \|g(x^{\ell_\omega})\|
\overset{\mathrm{\eqref{eq:Gamma-kappa-boundNORM2}}}{\leq}
\Delta_{\angle}^{-1} C_d^0 \sqrt{L\omega}$ and $C_d^0$ is independent of $\omega$,
we obtain $\|g(x^{\ell^{**}})\| =
\Oc\big(\Delta_{\angle}^{-1}\sqrt{L\omega}\big)$; see Remark~\ref{rem:asymptotic-convention}. Applying this bound to the following arguments yields \eqref{e.diffconvex1} and \eqref{e.diffsconvex1}, which completes the proof:
 
\begin{enumerate}[label=(\roman*)]
	\item By convexity of $f$, $f(x^{\ell^{**}})-f(\ul x)
	\leq
	g(x^{\ell^{**}})^T(x^{\ell^{**}}-\ul x)$. Applying the Cauchy--Schwarz inequality and using compactness of
	$\mathcal L(x^0)$ gives
	\[
	f(x^{\ell^{**}})-\ul f
	\leq
	\|g(x^{\ell^{**}})\|
	\|x^{\ell^{**}}-\ul x\|
	\leq
	\ol r\,\|g(x^{\ell^{**}})\|.
	\]
	\item By $\mu_c$-strong convexity,
	\[
	f(x^{\ell^{**}})-\ul f
	\leq
	\frac{1}{2\mu_c}
	\|g(x^{\ell^{**}})\|^2
	\qquad \text{and}
	\qquad
	\|x^{\ell^{**}}-\ul x\|
	\leq
	\frac{1}{\mu_c}
	\|g(x^{\ell^{**}})\|.
	\]
\end{enumerate}


\subsection{Proof of Proposition \ref{lem:noisy-eff-academic}}\label{pPro3}
Let $\alpha=\alpha_{\trec}^{\ell}$. Using the bound $|g(x)^Tp|$ in \eqref{eq:upBound} and the angle
condition \eqref{eq:angle_condition}, we obtain
\[
\Delta_{\angle}\|g(x)\|\|p\|
\leq
|g(x)^Tp|
\leq
\gamma_e\alpha
\Big(\wt\beta+\frac{L}{2}\delta^2\Big)
+
\frac{2\omega}{\alpha}.
\]
Reordering this inequality yields
\begin{equation}\label{e.inequalityD2}
\|g(x)\|
\leq
K_1\alpha
+
K_2\frac{\omega}{\alpha},
\qquad
\text{where}
\qquad
K_2:=
\frac{2}{\Delta_{\angle}\delta}.
\end{equation}
Equivalently, $K_1\alpha^2-\|g(x)\|\alpha+K_2\omega\geq0$. For this, $\alpha$ must not lie strictly
between the roots of the corresponding equation. The roots of $K_1\alpha^2-\|g(x)\|\alpha+K_2\omega=0$ are
\begin{equation}\label{e.roots2}
\alpha_{\pm}
=
\frac{\|g(x)\|\pm\sqrt{\Delta}}{2K_1},
\qquad
\text{where}
\qquad
\Delta
=
\|g(x)\|^2-4K_1K_2\omega.
\end{equation}
Now, we consider the following cases for $\Delta$ in \eqref{e.roots2}. If $\Delta>0$, then
\[
\alpha_-
=
\frac{\|g(x)\|-\sqrt{\Delta}}{2K_1}
<
\alpha_+
=
\frac{\|g(x)\|+\sqrt{\Delta}}{2K_1}.
\]
The quadratic inequality \eqref{e.inequalityD2} holds only if
$\alpha\leq\alpha_-$ or $\alpha\geq\alpha_+$. Rewriting the smaller root $\alpha_-$ and 
using 
\[
\frac{K_2}{K_1}
=
\frac{2\gamma_e^{-1}}
{(\wt\beta+L\delta^2/2)}
\leq
\frac{2}{\wt\beta},
\]
gives
\begin{equation}\label{e.alpha1}
\begin{aligned}
\alpha_-
&=
\frac{4K_1K_2\omega}
{2K_1\left(
\|g(x)\|
+
\sqrt{\|g(x)\|^2-4K_1K_2\omega}
\right)}
<
\sqrt{\frac{K_2}{K_1}\omega}
\leq
\sqrt{\frac{2\omega}{\wt\beta}}.
\end{aligned}
\end{equation}
Since $\alpha=\alpha_{\trec}^{\ell}$ is obtained from extrapolation, $\alpha
\geq
\alpha^{\ell0}
\geq
\alpha_{\min}$. Together with \eqref{eq:noise-upper-bound},  
\begin{equation}\label{eq:alpha-noise-lower}
\alpha
\geq
\alpha_{\min}
>
\sqrt{{2\omega}/{\wt\beta}}.
\end{equation}
Combining \eqref{e.alpha1} with \eqref{eq:alpha-noise-lower} gives
$\alpha>\alpha_-$. Therefore, the branch
$\alpha\leq\alpha_-$ is impossible, and
\eqref{e.inequalityD2} can hold only on the second feasible branch
$\alpha\geq\alpha_+$. Hence,
\[
\alpha
\geq
\alpha_+
=
\frac{
\|g(x)\|
+
\sqrt{\|g(x)\|^2-4K_1K_2\omega}
}{2K_1}
\geq
\frac{\|g(x)\|}{2K_1}.
\]
If $\Delta\leq0$, then $\|g(x)\|
\leq
2\sqrt{K_1K_2\omega}$ and therefore
\[
\frac{\|g(x)\|}{2K_1}
\leq
\sqrt{\frac{K_2}{K_1}\omega}
\leq
\sqrt{\frac{2\omega}{\wt\beta}}
<
\alpha,
\]
where the last inequality follows from
\eqref{eq:alpha-noise-lower}. Hence, in both cases,
\begin{equation}\label{e.alpha2_bound_final}
\alpha^2
\geq
\frac{1}{(2K_1)^2}\|g(x)\|^2.
\end{equation}
Combining \eqref{e.alpha2_bound_final} with Lemma~\ref{lem:noisy-to-noiseless} (under \texttt{NDF-SDC}) yields \eqref{neweffangle}.


\subsection{Proof of Theorem \ref{t.comf-noise1}}\label{pThe4}

Let $f^\ell:=f(x^\ell)$ and $g^\ell:=g(x^\ell)$ for simplicity. Moreover, let $S$ and $U$ be the sets of successful and
unsuccessful iterations from the set of pre-termination iterations $\mathcal K_T^- := \{\ell\in\mathbb N_0:0\leq\ell<\ell_T\}$, respectively, i.e., $\mathcal K_T^- = S \cup U$,  $S \cap U = \emptyset$ and $\ell_T = |\mathcal K_T^-| =|S|+|U|$. Let us also define  $\eta:={\beta}/{(2K_1)^2}$ in \eqref{neweffangle}. We include the Lipschitz constant $L>0$ in our conclusions; see Remark~\ref{rem:asymptotic-convention}.

We obtain the number of unsuccessful iterations before termination. If $\ell_\omega < \ell_T$, then $\ell_\omega \in \mathcal K_T^-$. Thus, at iteration $\ell = \ell_\omega$, equations \eqref{eq:scale_crossing_interval}, \eqref{eq:Gamma-kappa-boundNORM2}, and \eqref{e.varomega} imply $\|g^{\ell_\omega}\| \le \Delta_{\angle}^{-1} C_d^0 \sqrt{L\omega} \le C_d\sqrt{L\omega} = \varepsilon_\omega$, which contradicts the pre-termination condition $\|g^{\ell_\omega}\| > \varepsilon_\omega$. Therefore, we must have $\ell_T \le \ell_\omega$. Since the step-size scale $a^\ell$ is reduced from its initial value $a^0$ strictly during unsuccessful iterations via \eqref{eq:unsuccessful-sigma-contraction} as $a^{\ell+1} = \rho_u a^\ell$, and since $a^\ell > a_\omega$ for all $\ell < \ell_T$, we obtain the following bound
\begin{equation}\label{eq:unsuccessful-count}
|U| \le N_u := \left\lceil \log_{1/\rho_u} \left(\frac{a_0}{a_\omega}\right) \right\rceil, \qquad |U| = \Oc\left(L\log(\varepsilon_\omega^{-1}) \right).
\end{equation}
The second result follows from $a_\omega$ and $\varepsilon_\omega$ in \eqref{eq:a_omega} and \eqref{e.varomega}, 
$a_\omega=({\kappa_\star}/{(C_dL)})\varepsilon_\omega$. 

Since $\ell_T = |S| + |U|$, it remains only to bound $|S|$. Let $S:=\{\ell_0,\ldots,\ell_{|S|-1}\}$, with 
$\ell_0<\cdots<\ell_{|S|-1}$, as well as $x^j:=x^{\ell_j}$, $f^j:=f(x^j)$ and $g^j:=g(x^j)$ for $j=0,\ldots,|S|-1$, and $x^{|S|}:=x^{\ell_T}$, $f^{|S|}:=f(x^{\ell_T})$, and $g^{|S|}:=g(x^{\ell_T})$. 

\begin{enumerate}[label=(\roman*), itemsep=1ex]

\item For $\ell\in S$ before termination, 
$\|g^\ell\|>\varepsilon_\omega$. Hence, $f^\ell-f^{\ell+1}
\geq
\eta\|g^\ell\|^2
>
\eta\varepsilon_\omega^2$ by \eqref{neweffangle}. Now, summing over $\ell<\ell_T$ gives
\[
f^0-f^{\ell_T}
=
\sum_{\ell=0}^{\ell_T-1}(f^\ell-f^{\ell+1})
=
\sum_{\ell\in S}(f^\ell-f^{\ell+1})
\geq
\eta\varepsilon_\omega^2|S|, \qquad {\eta} \propto L^{-2}\Delta_{\angle}^2.
\]
Since $f\geq\ul f$, we obtain $|S|
\leq
\big(f^0-\ul f\big)
\big(\eta\varepsilon_\omega^2\big)^{-1}$, yielding $|S|=\Oc(L^2\varepsilon_\omega^{-2})$. Finally,
\[
\ell_T
=
|S|+|U|
=
\Oc(L^2\varepsilon_\omega^{-2})
+
\Oc(L\log(\varepsilon_\omega^{-1}))
=
\Oc(L^2\varepsilon_\omega^{-2}).
\]
\item Let $d_{\max}
:=
\sup_{x\in\mathcal L(x^0)}
\|x-\ul x\|$. By convexity and the Cauchy--Schwarz inequality,
\begin{equation}\label{e.conv}
f^j-\ul f
\leq
d_{\max}\|g^j\|,
\qquad
\text{for}
\quad
j=0,\ldots,|S|.
\end{equation}
Combining \eqref{e.conv} with \eqref{neweffangle} for
$j=0,\ldots,|S|-1$ yields $f^{j+1}
\leq
f^j
-
\eta d_{\max}^{-2}
(f^j-\ul f)^2$ which is equivalent to $\psi_{j+1}^{-1}-\psi_j^{-1}\geq\eta d_{\max}^{-2}$, where $\psi_j:=f^j-\ul f$ and 
$\psi_{j+1}\leq\psi_j$. 
Summing this inequality from $j=0$ to $k-1$, simplifying the telescoping sum and rearranging terms yields the optimality gap of the $k$-th successful iteration as
\begin{equation}\label{e.sublinEstim}
\psi_k
\leq
\left(
\eta d_{\max}^{-2}k+\psi_0^{-1}
\right)^{-1},
\qquad
k=1,\ldots,|S|.
\end{equation}
In particular, setting $k = |S|$ in \eqref{e.sublinEstim} bounds the optimality gap in terms of the total number of successful iterations. On the other hand, at the terminal iterate $x^{|S|}\in\mathcal L(x^0)$, where
$\|g^{|S|}\|\leq\varepsilon_\omega$, \eqref{e.conv} yields $\psi_{|S|}
\leq
d_{\max}\varepsilon_\omega$. Let $j_1 \in \{0, \dots, |S|\}$ be the first index satisfying 
$\psi_{j_1}\leq d_{\max}\varepsilon_\omega$. Achieving this requires $1/k = \Oc(\varepsilon_\omega)$, by \eqref{e.sublinEstim}, which yields $j_1 = \Oc(\varepsilon_\omega^{-1})$. Let us now discuss $j_1$.

If $j_1=|S|$, there are no successful iterations after $j_1$; i.e., $|S| = \Oc(\varepsilon_\omega^{-1})$. Otherwise, the algorithm requires additional iterations until $\|g^{|S|}\|\le\varepsilon_\omega$. If
$j_1<|S|$, then $\|g^j\|>\varepsilon_\omega$ for $j=j_1,\ldots,|S|-1$. For each such $j$, by \eqref{neweffangle}, $
f^j-f^{j+1}
\geq
\eta\varepsilon_\omega^2$. However, $f^{j_1}-f^{|S|} \leq f^{j_1}-\ul f \leq d_{\max}\varepsilon_\omega$ specifies the total decrease after $j_1$. Hence, the number of successful iterations after $j_1$ is
\[
\frac{\psi_{j_1}}{\psi_j-\psi_{j+1}} \leq
\frac{
d_{\max}\varepsilon_\omega
}{
\eta\varepsilon_\omega^2
}
=
\Oc(L^2\varepsilon_\omega^{-1}), \qquad {\eta} \propto L^{-2}\Delta_{\angle}^2.
\]
Therefore, combining $j_1 = \Oc(\varepsilon_\omega^{-1})$ and $|S| - j_1 = \Oc(\varepsilon_\omega^{-1})$ yields $|S| = \Oc(L^2\varepsilon_\omega^{-1})$. Finally, in both cases $j_1=|S|$ and $j_1<|S|$,
\[
\ell_T
=
|S|+|U|
=
\Oc(L^2\varepsilon_\omega^{-1})
+
\Oc(L\log(\varepsilon_\omega^{-1}))
=
\Oc(L^2\varepsilon_\omega^{-1}).
\]
\item The $\mu_c$-strong convexity bound $f^j-\ul f\leq(2\mu_c)^{-1}\|g^j\|^2$ yields $\|g^j\|^2
\geq
2\mu_c\psi_j$ where $\psi_j:=f^j-\ul f$ is the optimality gap for every successful
index $j=0,\ldots,|S|-1$. Combining this inequality with \eqref{neweffangle}  yields $\psi_{j+1} \le (1 - 2\eta\mu_c)\psi_j$. To ensure a contraction factor $q(\eta):= 1-2\eta\mu_c\in(0,1)$, we define $\widehat\eta := \min\{\eta,(4\mu_c)^{-1}\}$ so that $q_0:=q(\widehat\eta)\in[1/2,1)$. Since $\widehat\eta\leq\eta$, we obtain $\psi_{j+1}
\leq q_{0}\psi_j$, and hence $\psi_j\leq q_{0}^j\psi_0$. On the other hand, setting $x=x^j$ and
$\hat s=-L^{-1}g^j$ in \eqref{e.TayErr1}, together with
$f\geq\ul f$, yields $\|g^j\|^2
\leq
2L\psi_j$. Hence, the stationarity criterion $\|g^j\|\leq\varepsilon_\omega$ is
guaranteed whenever
$\psi_j\leq\varepsilon_\omega^2/(2L)$. Consequently, it suffices that
 
\[
|S|
\leq
\left\lceil
{\Big(
2\log(\varepsilon_\omega^{-1})
+
\log(2L\psi_0)\Big)
}/{
\Big(-\log(q_{0})\Big)
}
\right\rceil, \qquad  {-\log(q_0)} \propto L^{-2}\Delta_{\angle}^2.
\]
Therefore,
$|S|=\Oc\left(L^2\log(\varepsilon_\omega^{-1})\right)$. Finally,
$$\ell_T
=
|S|+|U|
=
\Oc\left(
L^2\log(\varepsilon_\omega^{-1})
\right) +
\Oc(L\log(\varepsilon_\omega^{-1}))
= \Oc(L^2\log(\varepsilon_\omega^{-1})).$$ 
\end{enumerate}

Now, we establish a uniform bound on evaluations per outer iteration of {\tt DAES}. By Proposition~\ref{prop:termination}, the extrapolation phase needs at most $N_e \le \lceil \log_{\gamma_e}(\alpha_{\max}/\alpha_{\min}) \rceil$ evaluations (independent of $\varepsilon_\omega$). Each iteration $\ell$ requires at most $\lambda N_{sel} + 2 + N_e$ constant evaluations (covering mutation--selection, initial signed search, and extrapolation), yielding $N_f = \Oc(\ell_T)$ and completing the proof.

\begingroup
\setlength{\bibsep}{2.pt}  
\bibliography{MathProg_Reference}
\endgroup

\end{document}